\documentclass[11pt]{article}
\usepackage{indentfirst, latexsym,bm}
\usepackage{amsmath}
\usepackage{pifont}
\usepackage{amsfonts}
\usepackage{mathrsfs}
\usepackage{array}
\usepackage{multirow} 
\usepackage{graphicx}
\usepackage{subfigure}
\usepackage{picinpar}
\RequirePackage[colorlinks,citecolor=blue,urlcolor=blue,linkcolor=blue]{hyperref} 

\usepackage{authblk}

\RequirePackage[numbers]{natbib}
\usepackage{enumitem} 

\usepackage{booktabs} 
\usepackage{threeparttable}
\usepackage{mathrsfs,amsfonts,amsmath,amssymb} 
\usepackage{color}

\makeatletter
\def\namedlabel#1#2{\begingroup
	#2%
	\def\@currentlabel{#2}%
	\phantomsection\label{#1}\endgroup
}
\makeatother

\newcommand\email[1]{\href{mailto:#1}{ \nolinkurl{#1}}}

\renewcommand{\theequation}{\arabic{section}.\arabic{equation}}

\newtheorem{theorem}{Theorem}[section]
\newtheorem{definition}[theorem]{Definition}
\newtheorem{lemma}[theorem]{Lemma}
\newtheorem{corollary}[theorem]{Corollary}
\newtheorem{proposition}[theorem]{Proposition}
\newtheorem{remark}[theorem]{Remark}
\newtheorem{condition}[theorem]{Condition}
\newtheorem{example}{Example}[section]

\def\blemma{\begin{lemma}}
	\def\elemma{\end{lemma}}
\def\bproposition{\begin{proposition}}
	\def\eproposition{\end{proposition}}
\def\ttheorem{\begin{theorem}}
	\def\etheorem{\end{theorem}}
\def\bcorollary{\begin{corollary}}
	\def\ecorollary{\end{corollary}}
\def\bremark{\begin{remark}}
	\def\eremark{\end{remark}}
\def\bcondition{\begin{condition}}
	\def\econdition{\end{condition}}

\newcommand{\opnorm}[1]{{\vert\kern-0.25ex\vert\kern-0.25ex\vert #1 
		\vert\kern-0.25ex\vert\kern-0.25ex\vert}}

\def\benumerate{\begin{enumerate}}
	\def\eenumerate{\end{enumerate}}
\def\bitemize{\begin{itemize}}
	\def\eitemize{\end{itemize}}

\def\beqlb{\begin{eqnarray}}
	\def\eeqlb{\end{eqnarray}}
\def\beqnn{\begin{eqnarray*}}
	\def\eeqnn{\end{eqnarray*}}
\def\ar{\!\!\!&}

\def\proof{\noindent{\it Proof.~~}}
\def\qed{\hfill$\Box$\medskip}

\begin{document}


 \title{\bf \Large  Mean-Field Limits for Nearly Unstable Hawkes Processes}
 \author{Gr\'egoire Szymanski\footnote{DMATH, Universit´e du Luxembourg, Email: gregoire.szymanski@uni.lu}
 \quad and \quad  
 Wei Xu\footnote{School of Mathematics and Statistics, Beijing Institute of Technology, Email: xuwei.math@gmail.com}}
 
 \date{}
 
 \maketitle
\begin{abstract} 
In this paper, we establish general scaling limits for nearly unstable Hawkes processes in a mean-field regime by extending the method introduced by Jaisson and Rosenbaum \cite{JaissonRosenbaum2016}. Under a mild asymptotic criticality condition on the self-exciting kernels $\{\phi^n\}$, specifically $\|\phi^n\|_{L^1} \to 1$, we first show that the scaling limits of these Hawkes processes are necessarily stochastic Volterra diffusions of affine type. Moreover, we establish a propagation of chaos result for Hawkes systems with mean-field interactions, highlighting three distinct regimes for the limiting processes, which depend on the asymptotics of $n(1-\|\phi^n\|_{L^1})^2$. These results provide a significant generalization of the findings by Delattre et al. \cite{DelattreFournierHoffmann2016}.
\end{abstract}

\medskip
 
\noindent \textbf{\textit{MSC2020 subject classifications.}}  Primary 60F05, 60G55, 60G22; secondary 60F17, 60G57.


  \section{Introduction} 
 \label{Sec.Introduction}
 \setcounter{equation}{0}

\subsection{Motivation}

A $d$-variate Hawkes process $N := (N_1, \ldots, N_d)$, when $d \geq 1$, is a $d$-dimensional time-inhomogeneous Poisson process, where the intensity $\lambda$ depends linearly on the past of the process $N$. More precisely, its intensity at time $t$, denoted by $\lambda(t) := (\lambda_1(t), \ldots, \lambda_d(t))$, is defined as
\begin{equation}
\label{eqn.Intensity}
\lambda_i(t) := \mu_i + \sum_{j=1}^d \sum_{\tau_{j,k} < t} \phi_{ij}(t-\tau_{j,k}) = \mu_i + \sum_{j=1}^d \int_{(0,t)} \phi_{ij}(t-s) \, dN_j(s), \quad t \geq 0, \, i = 1, \ldots, d,
\end{equation}
where the \textsl{background rate} $\mu \in \mathbb{R}_+^d$ represents the external excitation, and the \textsl{kernel} $\phi \in L^1_{\rm loc}(\mathbb{R}_+; \mathbb{R}_+^{d \times d})$ captures the self-/mutual excitation effects between components. Here, $\tau_{j,k}$ denotes the arrival time of the $k$-th event of type $j$. Given a background rate $\mu$ and a kernel $\phi$, a Hawkes process can be uniquely characterized in distribution (see for instance \cite{BremaudMassoulie1996}).\\

Initially introduced by G. Hawkes \cite{Hawkes1971a,Hawkes1971b} to study interactions between earthquakes and their aftershocks, Hawkes processes have since been applied in a wide range of fields, including neuroscience \cite{GrunRotter2010, OkatanWilsonBrown2005}, genome analysis \cite{Reynaud-BouretSchbath2010}, social sciences \cite{LewisMohlerBrantinghamBertozzi2011,Mohler2011}, and interactions in social networks \cite{ZhouZhaSong2013}. Hawkes-based models have also gained popularity in finance, where they play a key role in risk estimation \cite{Chavez-DemoulinDavisonMcNeil2005}, credit modeling \cite{GieseckeGoldbergDing2009}, and market microstructure modeling \cite{BacryDelattreHoffmannMuzy2013a, BacryDelattreHoffmannMuzy2013, BacryMuzy2014, Hewlett2006}. Their success is primarily due to two factors: \textsl{ease of use} and \textsl{flexibility}. More specifically, Hawkes-based models not only provide clear and interpretable dynamics, but many quantities of interest can also be computed in closed form. Furthermore, these models can be extended to accommodate complex stochastic systems, incorporating features such as self-inhibition \cite{CostaGrahamMarsalleTran2020}, multi-dimensional structures \cite{DelattreFournierHoffmann2016}, non-linearity \cite{Zhu2013}, or quadratic feedback mechanisms \cite{BlancDonierBouchaud2017}. For a comprehensive review of Hawkes processes and their various applications, we refer to \cite{BacryMastromatteoMuzy2015}.

In many cases, Hawkes processes are used to model microstructural behaviors, such as the arrival of orders in financial markets or the spiking of single neurons in neuroscience. A natural question is how these microstructural descriptions translate to macroscopic scales after appropriate scaling. Such limits can provide valuable insights into the macroscopic properties of these complex systems. Early studies on the long-term behavior of Hawkes processes primarily focused on their \textsl{stability}. A Hawkes process $N$ is considered \textsl{stable} (in the sense that it converges to a stationary state; see \cite[Definition 1]{BremaudMassoulie1996} for more details) if $\rho( \|\phi\|_{L^1}) < 1$, where
\begin{equation}\label{Con.Stability}
\|\phi\|_{L^1} := \int_0^\infty \phi(t) \, dt = \bigg( \int_0^\infty \phi_{ij}(t) \, dt \bigg)_{1\leq i,j \leq d},
\end{equation}
and $\rho(A)$ denotes the spectral radius of the matrix $A$. This condition is commonly referred to as the \textsl{stability condition}. It was first identified in \cite{HawkesOakes1974} using the cluster representation of Hawkes processes, along with the criticality of discrete-state branching processes. The result was later extended to the non-linear case in \cite{BremaudMassoulie1996} and to the univariate case with $\|\phi\|_{L^1} \leq 1$ and $\mu = 0$ in \cite{BremaudMassoulie2001}.

Under the stability condition, Bacry, Delattre, Hoffmann, and Muzy \cite{BacryDelattreHoffmannMuzy2013} proved a functional law of large numbers for $N$, stating that
\begin{equation*}
 	\sup_{0 \leq t \leq 1} \Big| \frac{N(Tt)}{T} -c_0 \cdot t \Big| \overset{\rm a.s.}\to 0
\end{equation*}
where $c_0 := \big(\mathbf{I} - \|\phi\|_{L^1}\big)^{-1} \mu$ and $\mathbf{I}$ is the identity matrix. Moreover, under an additional light-tailed condition,
\begin{equation}\label{eqn.LightTail}
\int_0^\infty \sqrt{t} \phi(t)\, dt < \infty,
\end{equation}
they also established a functional central limit theorem,
\begin{equation*}
\sqrt{T} \left( \frac{N(Tt)}{T} - c_0 \cdot t \right) \overset{d}{\to} c_1 \cdot B(t) \quad \text{as } T \to \infty,
\end{equation*}
where $B$ is a standard $d$-dimensional Brownian motion. Recently, Horst and Xu \cite{HorstXu2023} generalized these results in the univariate case by relaxing condition \eqref{eqn.LightTail}. Additionally, they established a functional central limit theorem for an unstable Hawkes process $N$ with $\|\phi\|_{L^1} = 1$, where the limiting process is a Gaussian process with long-range dependence.
 
The condition $\|\phi\|_{L^1} < 1$ can be relaxed by considering a sequence of nearly unstable Hawkes processes $N^n$, where the stability condition of \cite{BremaudMassoulie1996} is asymptotically violated. In \cite{JaissonRosenbaum2015, JaissonRosenbaum2016}, Jaisson and Rosenbaum consider a sequence of Hawkes processes $(N^n, \lambda^n)$ with kernel $\phi^n = a_n \cdot \phi$ for some $\|\phi\|_{L^1} = 1$ and $a_n < 1$. They construct scaling limits that converge to a non-negative limit process, exhibiting both affine and self-exciting properties. The limit process differs significantly depending on whether the kernel is light-tailed or heavy-tailed:

\begin{itemize}
    \item If $\int_0^\infty t \phi(t) \, dt < \infty$, the rescaled intensity converges weakly, as $a_n \to 1$, to a Cox-Ingersoll-Ross (CIR) process $\xi$ defined by  
    \begin{equation}\label{eq:volerra:limit:intro0}
    \xi(t) =  \int_0^t b\big(a - \xi(s)\big) \, ds +  \int_0^t \sigma\sqrt{\xi(s)} \, dB(s),\quad t\geq 0.
    \end{equation} 
    Moreover, the rescaled Hawkes process converges weakly to the integrated process of $\xi$.
    
    \item If $\int_0^\infty t \phi(t) \, dt = \infty$ and $\int_T^\infty \phi(t) \, dt \sim C \cdot T^{-\alpha}$ as $T\to\infty$ for some $\alpha \in (1/2, 1)$, then the rescaled Hawkes process converges weakly, as $a_n \to 1$, to the integrated process of a non-negative process $\xi$ that solves the following stochastic Volterra equation:  
    \begin{equation}
    \label{eq:volerra:limit:intro}
    \xi(t) =  \int_0^t \frac{(t-s)^{\alpha-1}}{\Gamma(1-\alpha)} \cdot b\big(a - \xi(s)\big) \, ds +  \int_0^t \frac{(t-s)^{\alpha-1}}{\Gamma(1-\alpha)} \cdot \sigma\sqrt{\xi(s)} \, dB(s),\quad t\geq 0. 
    \end{equation}
    The weak convergence of the rescaled intensity to $\xi$ was established in a recent work \cite{HorstXuZhang2023a}.
\end{itemize}
 
Analogous results for the multivariate case are established in \cite{ElEuchFukasawaRosenbaum2018, RosenbaumTomas2021}. These limit results are of great interest in mathematical finance, as they provide a microeconomic foundation for the Heston model and link Hawkes processes to the theory of rough volatility introduced in \cite{GatheralJaissonRosenbaum2018}. Recently, they have been adapted to study the properties of rough models in \cite{ElEuchRosenbaum2019b} and to explain market impact characteristics in \cite{durin2023two,JusselinRosenbaum2020}. These results have also been generalized to quadratic Hawkes processes in \cite{DandapaniJusselinRosenbaum2021}, establishing a connection between quadratic Hawkes models and rough volatility models to study the Zumbach effect.

All the aforementioned results mainly focus on the long-term behavior of Hawkes processes, which is particularly relevant for financial applications. However, different asymptotic regimes can also be considered in various contexts. For instance, as an important application in neuroscience, Hawkes processes have been used to model neuron interactions, and their mean-field limits are widely accepted and considered more appropriate in this setting. Typically, we are interested in a particle system consisting of $n$ exchangeable Hawkes processes $N^n_1, \dots, N^n_n$ sharing a common intensity 
 \begin{equation}
 \label{eq:intensity:mean_field}
 \lambda^n_i(t) = \lambda^n(t) = \frac{1}{n} \bigg( \mu + \sum_{j=1}^n \int_0^t \phi(t-s) \, dN^n_j(s) \bigg).
 \end{equation}
 Under some non-degeneracy conditions,  Delattre, Fournier and Hoffmann \cite{DelattreFournierHoffmann2016} proved a propagation of chaos as $n \to \infty$. In this regime,  particles become  mutually independent asymptotically and behave as a sequence of orthogonal inhomogeneous Poisson processes with common intensity given by the unique solution of the Volterra integral equation 
 \begin{equation*}
 \overline{\lambda}_t = \mu + \int_0^t \phi(t-s) \overline{\lambda}_s \, ds.
 \end{equation*}
 Later, this result was generalized to age-dependent processes in \cite{chevallier2017mean} and to space-dependent networks in \cite{chevallier2019mean}, where more general limit results were obtained. For age-dependent Hawkes processes, the mean-field limit can be expressed as a system of partial differential equations, similar to those introduced in \cite{pakdaman2009dynamics}, while space-dependent networks converge to a CIR-type diffusion. Other processes related to Hawkes processes specific to neuroscience have been studied in \cite{erny2021conditional, pfaffelhuber2022mean}. Fluctuations and deviations around these limits have also been explored in \cite{GaoGaoZhu2023,GaoZhu2023}.

\subsection{Main results and organization of the paper}

 The main objective of this work is twofold. 
 Firstly, we generalize the results in \cite{JaissonRosenbaum2015, JaissonRosenbaum2016} by considering more flexible assumptions on the self-exciting kernel $\phi^n$. 
 Specifically, we remove the assumption that $\phi^n = a_n \phi$ and instead consider general kernels $\phi^n$ satisfying $\|\phi^n\| \to 1$.
 This allows us to derive more general limiting stochastic Volterra processes that belong to the class studied in \cite{jaber2019affine}.
 In particular, our limit models in the univariate case uniquely solve stochastic Volterra equations in the form of 
 \beqlb
 \label{eq:new_volerra:limit:intro}
\xi(t) = F(t)\cdot a + \int_0^t f(t-s) \cdot \sigma\sqrt{\xi(s)} \, dB(s), \quad t \geq 0,
 \eeqlb
 that significantly extend the limit models \eqref{eq:volerra:limit:intro0}-\eqref{eq:volerra:limit:intro},
 where $a,\sigma>0$, $F$ is a probability distribution function on $\mathbb{R}_+$ with density $f$ uniquely determined by the Laplace transform 
 \beqnn
 \int_0^\infty e^{-zt} f(t)dt=	\Big(b + c\cdot z + \int_0^\infty (1-e^{-z x}) \nu(dx) \Big)^{-1},\quad z\geq 0,
 \eeqnn
 for some constants $b, c \geq 0$ and $\sigma$-finite measure $\nu(dx)$ on $(0, \infty)$ satisfying $\int_0^\infty (1 \wedge x) \, \nu(dx) < \infty$. We also establish analogous results for the multidimensional setting.

Secondly, we address the seemingly different problem of mean-field limits for interacting systems composed of Hawkes particles.
 Specifically, we consider a particle system of $n$ exchangeable Hawkes processes $N^n_1, \dots, N^n_n$ with a common intensity given by \eqref{eq:intensity:mean_field}, but the background rate $\mu$ and kernel $\phi$ are replaced by $\mu^n$ and $\phi^n$ respectively.
 Due to the symmetry among particles, the aggregate process $\sum_{k=1}^n N^n_k$ remains a Hawkes process. Our preceding limit theorems demonstrate that after an appropriate rescaling, it asymptotically behaves as described in \eqref{eq:new_volerra:limit:intro}. 
 Based on the important observation that removing a finite number of coordinates from the particle system \eqref{eq:intensity:mean_field} does not alter its asymptotic behavior, we first couple it with a new Hawkes particle system in which the common intensity of all particles inherits the perturbations of all particles except finite fixed ones in the original particle system. 
 Then we obtain the mean-field limits for \eqref{eq:intensity:mean_field} directly from those of the new system that are established by using a weak convergence result for empirical distributions of a sequence of interchangeable abstracted-valued random variables.

 The rest of this paper is structured as follows. 
 In Section~\ref{Sec.MainResultSL}, we formulate the results concerning the scaling limits of multidimensional theorem. 
 In Section~\ref{Sec.MainResultMF}, we provide the mean-field limits for Hawkes processes and state a propagation of chaos result.   
 All results are proved one-by-one in Sections~\ref{Sec.Proof01}-\ref{Sec.Proof03}. 
 A summary about Hawkes processes and some additional statements regarding the regularity of our limit processes are presented in Appendices~\ref{Sec.AppendixHawkes} and \ref{Sec.AppendixHCY} respectively.

\subsection{Notation}

Let $\mathbb{B}$ be a metric space with metric $d$, and let $m, n \in \mathbb{Z}_+$. We denote by $\mathbb{B}^{m \times n}$ the space of all $m \times n$ matrices $A$ with elements $A_{ij} \in \mathbb{B}$ for $1 \leq i \leq m$ and $1 \leq j \leq n$. For $p \in (0, \infty]$ and an interval $I \subset \mathbb{R}_+$, let $L^p(I; \mathbb{B})$ denote the space of all $p$-integrable $\mathbb{B}$-valued functions $g$ defined on $I$, equipped with the norm 
\[
\|g\|_{L^p_I} := \left( \int_I d(g(x), 0)^p \, dx \right)^{1/p}.
\]
We also write $\|g\|_{L^p_T} = \|g\|_{L^p_{[0,T]}}$ and $\|g\|_{L^p} = \|g\|_{L^p_{[0,\infty)}}$. Let $L^p_{\rm loc}(\mathbb{R}_+; \mathbb{B})$ be the space of all functions $g$ such that $g \in L^p([0, T]; \mathbb{B})$ for any $T \geq 0$.

Let $\mathrm{M}(\mathbb{B})$ denote the space of all finite measures $\nu$ on $\mathbb{B}$. For a function $g$ defined on $\mathbb{B}$, whenever it is well-defined, we use the notation
\[
\nu(g) := \int_\mathbb{B} g(x) \, \nu(dx).
\]
We endow the space $\mathrm{M}(\mathbb{B})$ with the weak topology, meaning that $\nu_n \to \nu$ if $\nu_n(g) \to \nu(g)$ for any bounded and continuous function $g$ on $\mathbb{B}$.

The convolution of two functions $g, h \in L^1_{\rm loc}(\mathbb{R}_+; \mathbb{R})$ is defined as
\[
(g * h)(t) := \int_0^t g(t-s) h(s) \, ds, \quad t \geq 0.
\]
We denote by $g^{*k}$ the $k$-th convolution of $g$ with itself. For $g \in L^1_{\rm loc}(\mathbb{R}_+; \mathbb{R}^{m \times l})$ and $h \in L^1_{\rm loc}(\mathbb{R}_+; \mathbb{R}^{l \times n})$, their convolution is defined as
\[
(g * h)_{ij} := \sum_{k=1}^l (g_{ik} * h_{kj}), \quad 1 \leq i \leq m, \, 1 \leq j \leq n.
\]

We also write $g\circ h$ for the composition of two functions $g$ and $h$, i.e., $g\circ h(t)=g(h(t))$, when this makes sense.

Throughout this paper, we use the generic constant $C$, which may vary from line to line.

\section{Scaling limits for multidimensional Hawkes processes}
\label{Sec.MainResultSL} 
\setcounter{equation}{0}
 
\subsection{Statement of the main results}

In this section, we establish a scaling limit theorem for a sequence of multivariate Hawkes processes, generalizing the results of \cite{JaissonRosenbaum2015, JaissonRosenbaum2016}.

For each $n \geq 1$, we consider a filtrated probability space $(\Omega, \mathscr{F}, \mathscr{F}_t, \mathbf{P})$ on which a $d$-dimensional Hawkes process $N^n$ is defined, with intensity $\lambda^n$ given by
\begin{equation*}
\lambda^n(t) = \mu^n + \int_0^t \phi^n(t-s) \, dN^n(s), \quad t \geq 0,
\end{equation*}
for some background rate $\mu^n \in \mathbb{R}_+^d$ and self-exciting kernel $\phi^n \in L^1_{\rm loc}(\mathbb{R}_+; \mathbb{R}_+^{d \times d})$. We refer to Appendix~\ref{Sec.AppendixHawkes} for a precise definition of Hawkes processes (see Definition~\ref{def:hawkes}).

Let $\psi^n$ denote the resolvent of the second kind associated with $\phi^n$, defined as
\begin{equation*}
\psi^n = \sum_{k=1}^{\infty} (\phi^n)^{*k}.
\end{equation*}
 We denote by $\Lambda^n$ the compensator of $N^n$ and define $M^n:= N^n - \Lambda^n$ that is the compensated point process associated with $N^n$. 
 Using the martingale representation \eqref{MartingaleRep}, we have
\begin{equation} \label{MartingaleRep.n}
\lambda^n(t) = \mu^n + \big\|\psi^n\big\|_{L^1_t} \cdot \mu^n + \int_0^t \psi^n(t-s) \, dM^n(s), \quad t \geq 0,
\end{equation}
from which we see that the long-term behavior of $(N^n, M^n)$ is fully determined by the asymptotic properties of $\psi^n$.

Different modes of convergence can be considered in this setup. In this work, we focus on weak convergence in the space $\mathrm{D}(\mathbb{R}_+; \mathbb{B})$ of all $\mathbb{B}$-valued c\`{a}dl\`{a}g functions on $\mathbb{R}_+$, endowed with the Skorokhod topology, or in the space $\mathrm{C}(\mathbb{R}_+; \mathbb{B})$ of all $\mathbb{B}$-valued continuous functions on $\mathbb{R}_+$, endowed with the uniform topology for some metric space $\mathbb{B}$ (see \cite{Billingsley1999, JacodShiryaev2003}). Recall that $\overset{\rm v}{\to}$ denotes vague convergence. To obtain a non-degenerate scaling limit, we assume the following $L^2$-condition and convergence condition hold for the resolvent sequence $\{ \psi^n \}_{n \geq 1}$.

\begin{condition} \label{Main.Condition.01}
Assume that there exists a positive sequence $\{ \beta_n \}_{n \geq 1}$ with $\beta_n \to 0$ as $n \to \infty$, such that for each $1 \leq i, j \leq d$ and $T \geq 0$, we have
\begin{equation*}
\sup_{n \geq 1} \, \beta_n \cdot \big\|\psi^n_{i,j}\big\|_{L^2_T} < \infty,
\end{equation*}
and there exists a $\sigma$-finite measure $F_{ij}(dt)$ on $\mathbb{R}_+$ such that as $n \to \infty$,
\begin{equation*}
F_{ij}^n(dt) := \beta_n \cdot \psi^n_{i,j}(t) \, dt \overset{\rm v}{\to} F_{ij}(dt).
\end{equation*}
\end{condition}

A direct consequence of Theorem~10 in \cite{MeyerZheng1984}, along with Condition~\ref{Main.Condition.01}, is that the $\sigma$-finite measure $F_{ij}(dt)$ is absolutely continuous with respect to the Lebesgue measure, with a density function $f_{ij} \in L^2_{\rm loc}(\mathbb{R}_+; \mathbb{R}_+)$. Therefore, for any $g \in \mathrm{C}(\mathbb{R}_+; \mathbb{R})$ and $T \geq 0$, we have
\begin{equation} \label{eqn.212}
\lim_{n\to\infty} \int_0^T g(t) \beta_n \psi^n_{i,j}(t) \, dt = \int_0^T g(t) f_{ij}(t) \, dt.
\end{equation}
For simplicity, we denote by $F_{ij}(t) := F_{ij}([0, t])$ the cumulative distribution function of $F_{ij}(dt)$.

We are now ready to state our first result concerning the weak convergence of the spatially-scaled processes $\{(\Lambda^{(n)}, N^{(n)}, M^{(n)}) \}_{n \geq 1}$, defined as
\begin{equation*}
\Lambda^{(n)}(t) := \beta_n^2 \cdot \Lambda^n(t), \quad
N^{(n)}(t) := \beta_n^2 \cdot N^n(t), \quad \text{and} \quad
M^{(n)}(t) := \beta_n \cdot M^n(t), \quad t \geq 0, \, n \geq 1.
\end{equation*}
The following result, proved in Section~\ref{Sec.Proof01}, details the convergence of these processes.

\begin{theorem} \label{Thm.MainResultSL.01}
If Condition~\ref{Main.Condition.01} holds and $\beta_n \cdot \mu^n \to a \in \mathbb{R}_+^d$ as $n \to \infty$, then the following hold.
\begin{enumerate}
    \item[(1)] The sequence of rescaled processes $\{(\Lambda^{(n)}, N^{(n)}, M^{(n)})\}_{n \geq 1}$ is $C$-tight.
    
    \item[(2)] For any limit $(U, X, Z) \in \mathrm{C}(\mathbb{R}_+; \mathbb{R}_+^d \times \mathbb{R}^d)$, we have $U \overset{\rm a.s.}{=} X$. Moreover, the process $X$ is non-decreasing, and there exists a $d$-dimensional Brownian motion $B = (B_1, \cdots, B_d)$ such that
    \begin{equation} \label{eqn.LimitSVE}
    Z(t) = \big(B_i \circ X_i(t)\big)_{1 \leq i \leq d} \quad \text{and} \quad X(t) = \big\|F\big\|_{L^1_t} \cdot a + f * Z(t), \quad t \geq 0.
    \end{equation}
\end{enumerate}
\end{theorem}

\begin{remark} \label{Remark.01}
Since $f \in L^2_{\rm loc}(\mathbb{R}_+; \mathbb{R}_+^{d \times d})$, an analogue to \cite[Lemma 2.1]{AbiJaber2021} implies that the process $X$ is differentiable, and its derivative $Y$ is a non-negative weak solution to the stochastic Volterra equation
\begin{equation*}
Y(t) = F(t) \cdot a + \int_0^t f(t-s) \, dZ(s), \quad t \geq 0.
\end{equation*}
 Note that $Y$ is continuous and hence also $(\mathscr{F}_t)$-predictable.
 Using \eqref{eqn.LimitSVE}, the preceding equation is equivalent to  
\begin{equation*}
Y(t) = F(t) \cdot a + \int_0^t f(t-s) \sqrt{{\rm diag}(Y(s))} \, dB(s), \quad t \geq 0.
\end{equation*}
It is typical that the regularity properties of $f$ are inherited by $Y$. For more details on the H\"older continuity of $Y$, we refer to Appendix~\ref{Sec.AppendixHCY}.
\end{remark}

\subsection{Characterization of $F$ and comparison with the existing literature}

In this section, we provide a partial characterization for the measure $F$ and its density $f$ provided in satisfying Condition~\ref{Main.Condition.01}. To avoid degeneracy of the model, we always assume the stability condition holds, i.e.,  the spectral radius $\rho(\|\phi^n\|_{L^1}) < 1$.  

We define the Laplace transforms of a function $g\in L^1(\mathbb{R}_+; \mathbb{R}_+ )$ and of a measure $\nu\in  \mathrm{M}(\mathbb{R}_+) $ by
 \beqnn
 \mathcal{L}_g(z):= \int_{\mathbb{R}_+} e^{-zt}g(t)dt
 \quad \mbox{and}\quad 
 \mathcal{L}_\nu(z):= \int_{\mathbb{R}_+} e^{-zt}\nu(dt),\quad z\geq 0.
 \eeqnn
Their Fourier transforms are defined respectively by $\mathcal{F}_g(z): =\mathcal{L}_g(\mathtt{i}z)$ and $\mathcal{F}_\nu(z): =\mathcal{L}_\nu(\mathtt{i}z)$ for $z\in\mathbb{R}$ where $\mathtt{i}$ is the imaginary unit. The Laplace/Fourier transforms of matrix-valued functions and measure-valued matrices are given by computing the Laplace/Fourier transform of each element.

\textbf{\textit{Univariate case: $d = 1$.}} The stability condition reduces to $\|\phi^n\|_{L^1} < 1$. A full characterization of $F(dx)$ can be obtained by considering its Laplace transform and using the theory of Bernstein functions. A function $g$ on $\mathbb{R}_+$ is called a \textsl{Bernstein function} if and only if it can be expressed as
\begin{equation*}
g(z) = b + c z + \int_0^\infty (1 - e^{-zx}) \, \nu(dx), \quad z \geq 0,
\end{equation*}
for some constants $b, c \geq 0$ and a $\sigma$-finite measure $\nu(dx)$ on $(0, \infty)$ such that $\int_0^\infty (1 \wedge x) \, \nu(dx) < \infty$. The triplet $(b, c, \nu)$ is referred to as the \textsl{L\'evy triplet} of $g$. We refer to \cite{SchillingSongVondracek2012} for a comprehensive review of Bernstein functions.

The stability condition ensures that $\|\psi^n\|_{L^1} < \infty$. Taking the Laplace transform on both sides of \eqref{Resolvent} or \eqref{eqn.Resolvent}, we obtain
\begin{equation*}
\mathcal{L}_{\psi^n}(z) = \frac{\mathcal{L}_{\phi^n}(z)}{1 - \mathcal{L}_{\phi^n}(z)} 
\quad \text{and} \quad 
\mathcal{L}_{F^n}(z) = \frac{\beta_n \cdot \mathcal{L}_{\phi^n}(z)}{1 - \mathcal{L}_{\phi^n}(z)}, \quad z \geq 0.
\end{equation*}
The vague convergence of $\mathcal{L}_{F^n}(dt)$ to $F(dt)$ is equivalent to the pointwise convergence of $\mathcal{L}_{F^n}$ to $\mathcal{L}_{F}$, which occurs if and only if, as $n \to \infty$,
\begin{equation} \label{eqnA.1}
\mathcal{L}_{\phi^n}(z) \to 1 
\quad \text{and} \quad 
\Phi^n(z) := \beta_n^{-1} \cdot \big( 1 - \mathcal{L}_{\phi^n}(z) \big) \to \Phi(z) \in [0, \infty], \quad z \geq 0.
\end{equation}
Note that in that case, we can express $\Phi^n(z)$ as
\begin{equation*}
\Phi^n(z) = \beta_n^{-1} \cdot \big( 1 - \|\phi^n\|_{L^1} \big) + \int_0^\infty (1 - e^{-zt}) \beta_n^{-1} \phi^n(t) \, dt,
\end{equation*}
which shows that $\Phi^n$ is a Bernstein function with L\'evy triplet $\left( \beta_n^{-1} \cdot \big( 1 - \|\phi^n\|_{L^1} \big), 0, \beta_n^{-1} \phi^n(t) \, dt \right)$. The following result is a direct consequence of Corollary~3.9 in \cite[p.29]{SchillingSongVondracek2012}.

\begin{lemma}
\label{lem:caracterisationphi}
\begin{enumerate}
    \item[(1)] The limit function $\Phi$ is a Bernstein function, and the limit measure $F$ is uniquely determined by the Laplace transform $\mathcal{L}_F(z) = 1 / \Phi(z)$ for $z \geq 0$.
    
    \item[(2)] For any Bernstein function $\Phi$, there exists a sequence of functions $\{\phi^n\}_{n \geq 1} \subset L^1(\mathbb{R}_+; \mathbb{R}_+)$ and a sequence $\{\beta_n \}_{n \geq 1}$ vanishing asymptotically, such that \eqref{eqnA.1} holds.
    
    \item[(3)] The vague convergence $F^n \to F$ is equivalent to $\mathcal{L}_{\phi^n}(z) = 1 - \beta_n / \mathcal{L}_F(z) + o(\beta_n)$ as $n \to \infty$.
\end{enumerate}
\end{lemma}

The limit measure $F$ corresponds to the resolvent measure of a killed subordinator with Laplace exponent $\Phi$ (see \cite{Bertoin1996}). By the Fourier isometry, the uniform $L^2$-upper bound in Condition~\ref{Main.Condition.01} is equivalent to $\sup_{n \geq 1} \|\mathcal{F}_{F^n}\|_{L^2} < \infty$. To the best of our knowledge, a necessary and sufficient condition on $\{\phi^n\}$ for this bound is still unknown. However, in dimension $d = 1$, we have by Young's convolution inequality
\beqnn
\big\|\psi^n\big\|_{L^2_T} \leq 
\big\|\phi^n\big\|_{L^2_T} \sum_{k=0}^{\infty} \big\|(\phi^n)^{*k}\big\|_{L^1_T} ,
\eeqnn
which implies that 
\beqnn
\big\|\psi^n\big\|_{L^2_T} \leq \frac{\big\|\phi^n\big\|_{L^2_T}}{1-\big\|\phi^n\big\|_{L^1_T}}.
\eeqnn
In particular, provided the second limit in \eqref{eqnA.1}, the $L^2$-condition is verified whenever $\big\|\phi^n\big\|_{L^2_T}$ is bounded independently of $n$.

It is worth noting that all Hawkes processes considered in \cite{HorstXuZhang2023a, JaissonRosenbaum2015, JaissonRosenbaum2016, Xu2021} can be seen as special cases of Lemma~\ref{lem:caracterisationphi}.
 We illustrate this by specifically examining the cases in \cite{JaissonRosenbaum2015} and \cite{JaissonRosenbaum2016}. In these papers, the authors take
\beqnn
\phi^n(t) = a_n b_n \cdot \phi(b_n t),
\eeqnn
where $(a_n)_n$ and $(b_n)_n$ are positive sequences such that
\beqnn
a_n \to 1, \quad b_n \to \infty \quad \text{and} \quad (1 - a_n) b_n \to c,
\eeqnn
for some $c > 0$, and where $\phi$ is a continuous bounded function such that $\|\phi\|_{L^1} = 1$. 
In that case, we have
\begin{equation*}
\psi^n(t) = b_n \sum_{k=1}^\infty a_n^{k} \phi^{*k}(b_n t), \quad t \geq 0.
\end{equation*}
 This implies, in particular, that $\|\psi^n\|_{L^1} = (1 - a_n)^{-1}$.
%
%
%
%
%
Taking $\beta_n = 1 - a_n \to 0$, \cite{JaissonRosenbaum2015} and \cite{JaissonRosenbaum2016} show that the limit of the measure $F^n(dt) = \beta_n \psi^n(t) \, dt$ depends on the tail behavior of $\phi$. More precisely, when
\begin{equation*}
\int_0^\infty s \phi(s) \, ds = \int_0^{\infty} \left( 1 - \int_0^t \phi(s) \, ds \right) dt < \infty,
\end{equation*}
the weak limit of $F^n(dt)$ is of the form
\begin{equation*}
\mathcal{L}_{F}(z) = \frac{1}{m +  \lambda\cdot z}
\end{equation*}
for some constants $m, \lambda > 0$. When
\begin{equation*}
1 - \int_0^t \phi(s) \, ds \sim C t^{-\alpha}
\end{equation*}
for some $1/2 < \alpha < 1$, we instead have
\begin{equation*}
\mathcal{L}_{F}(z) = \frac{1}{m + \lambda \cdot z^{\alpha}}
\end{equation*}
for some constants $m, \lambda > 0$.

 \textbf{\textit{Multivariate case: $d>1$.}}  
 In contrast to the univariate case, the characterization of $F$ with $d>1$ is much more intricate because the Fourier transforms involve matrices. 
 Here we just characterize the measure $F$ in several tractable settings. 
 First, we show that the Fourier transform of $\phi^{n}$ is always invertible under the stability condition. For $x\in \mathbb{R}^d$, let $\|x\|_{l^1}=\sum_{i=1}^d |x_i|$.
 For a matrix $A \in \mathbb{R}^{d\times d}$, let ${\rm Adj}(A)$ be its adjacent matrix, ${\rm Tr}(A)$ its trace and 
 \beqnn
 \opnorm{A}:=  \max_{1 \leq j \leq d } \sum_{i=1}^{d}| A_{i,j} | = \sup_{\|x\|_{l^1} = 1} \|Ax\|_{l^1}
 \eeqnn

 \begin{proposition} 
 For any  $g\in L^1(\mathbb{R}_+;\mathbb{R}^{d\times d})$, then $\rho \big( \mathcal{F}_g  (z) \big) \leq \rho(\|g\|_{L^2})$ for any $z\in\mathbb{R}$. 
 \end{proposition}
 \proof 
 Recall that $ \big| \big(\mathcal{F}_g  (z)\big)^p  \big| \leq  \big| \mathcal{F}_g (z) \big|^p $  and $\big| \mathcal{F}_g  (z) \big| \leq \big| \mathcal{F}_g (0) \big|\leq \| g \|_{L^1}$ component-wise for any $z\in\mathbb{R}$ and $p>0$. 
 Thus using the fact that $\opnorm{\cdot}$ is an operator norm, we have
 \beqnn
 \rho \big( \mathcal{F}_g  (z) \big) 
  = \lim_{p \to \infty} \opnorm{\big(\mathcal{F}_g  (z)\big)^p}^{1/p} 
  \leq \limsup_{p \to \infty} \opnorm{\big(\mathcal{F}_g  (0)\big)^p}_1^{1/p} 
  \leq \rho(\|g\|_{L^2}) .
 \eeqnn
 \qed
 
 Under the stability condition, the preceding proposition ensures the invertibility of the matrix $\mathbf{I}-\mathcal{F}_{\phi^n}(z)$ for any $z\in\mathbb{R}$ and $n\geq 1$.  
 Similarly to the univariate case, taking Fourier transforms on both sides of \eqref{Resolvent} or \eqref{eqn.Resolvent} we also have
 \beqlb\label{eqn.A03}
 \mathcal{F}_{\psi^n}(z)= \big( \mathbf{I}-\mathcal{F}_{\phi^n}(z) \big)^{-1}\cdot\mathcal{F}_{\phi^n}(z) 
 \quad \mbox{and}\quad
 \mathcal{F}_{F^n}(z)= \beta_n\big( \mathbf{I}-\mathcal{F}_{\phi^n}(z) \big)^{-1}\cdot\mathcal{F}_{\phi^n}(z),\quad t\geq 0.
 \eeqlb
 
We now present two settings in which we can explicitly compute the limiting measure $F$. Unlike the univariate case, we do not claim that these are the only two regimes satisfying Condition~\ref{Main.Condition.01}. A more comprehensive study is beyond the scope of this paper.  
 
 {\bf Setting~I.}  For $n\geq 1$, assume that 
 $ \mathcal{F}_{\phi^n}(z) = \mathbf{I} -\beta_n\cdot B( z )  + o ( \beta_n ) $
  for some $B \in \mathrm{C}(\mathbb{R};\mathbb{R}^{d\times d})$ with $B(z)$ being invertible for all $z \in\mathbb{R}$. 
  Plugging this back into \eqref{eqn.A03} and letting $n\to\infty$, we get
  \begin{align*}
 \mathcal{F}_{F^n}(z) =\big(B( z )  + o ( 1)\big)^{-1} \cdot \mathcal{F}_{\phi^n}(z) \to \big(B (z )\big) ^{-1} ,\quad z\in\mathbb{R}. 
  \end{align*}
Using L\'evy's continuity theorem, we conclude that the limit function is the Fourier transform of a non-negative measure $F$, and $F^n$ converges narrowly to $F$.
  
  {\bf Setting~II.} For $n\geq 1$, suppose that $ \mathcal{F}_{\phi^n}(z) = A(z) -\beta_n\cdot B( z )  + o ( \beta_n ) $ 
  where $A,B \in \mathrm{C}(\mathbb{R};\mathbb{R}^{d\times d})$ such that  $A(z) \neq \mathbf{I}$ for all $z \in \mathbb{R}$ and $A(z) -  \mathbf{I}$ is non-invertible.
  Suppose also that the function $z \mapsto {\rm Tr} ( {\rm Adj} ( \mathbf{I}-A(z) ) B (z ) )$ never vanishes.
   Plugging this back into \eqref{eqn.A03} and letting $n\to\infty$, we get
  \begin{align*}
  	\mathcal{F}_{F^n}(z) \to\frac{{\rm Adj}  ( \mathbf{I} - A(z) )}
   	{{\rm Tr}\big({\rm Adj} ( \mathbf{I} - A(z) ) B (z) \big)},\quad z\in\mathbb{R}. 
  \end{align*}
Using L\'evy's continuity theorem, we deduce that the measures $F^n$ converge narrowly to a non-negative measure $F$, whose Fourier transform is given by the above limit function.

Again, this aligns with the existing literature. \cite{ElEuchFukasawaRosenbaum2018} takes $\Phi^n(t)$ to be diagonalizable in a basis independent of $n$ and $t$, while \cite{RosenbaumTomas2021} studies dynamics similar to this, with $\Phi^n(t)$ being trigonalizable, still in a basis independent of $n$ and $t$. The limiting processes in both works slightly differ from those obtained in Theorem~\ref{Thm.MainResultSL.01}, as they focus on price processes obtained as linear transformations of the process $N^{(n)}$.

\section{Mean-field limits }

 \label{Sec.MainResultMF} 
 \setcounter{equation}{0}
  
In this section, we provide a propagation of chaos result for high-dimensional Hawkes processes with mean-field interactions, following the framework of \cite{DelattreFournier2016}. More precisely, for each $n \geq 1$, we assume that $N^n$ is an $n$-variate Hawkes process with parameters
 \beqnn
 \mu_i^n= \frac{\mu_0^n}{n}
 \quad \mbox{and}\quad 
  \phi^n_{ij}(t)= \frac{\varphi^n(t)}{n},\quad t\geq 0, \, 1\leq i,j\leq n,
 \eeqnn
for some constant $\mu_0^n \geq 0$ and function $\varphi^n \in L^1_{\rm loc}(\mathbb{R}_+; \mathbb{R}_+)$. The intensity $\lambda^n = (\lambda^n_1, \dots, \lambda^n_n)$ of $N^n$ is given by
 \beqlb
 	\lambda^n_i(t) = \lambda^n_0(t):=  \frac{\mu_0^n}{n}  +  \frac{1}{n}\sum_{j=1}^n \int_0^t \varphi^n(t-s)\, dN^n_j(s), \quad t\geq 0, \, 1\leq i\leq n.
 \eeqlb

 The point process $N^n$ can be seen as an interacting particle system with indistinguishable particles, in the sense that $(N^n_1, \cdots, N^n_n) \overset{\rm d}{=} (N^n_{\sigma(1)}, \cdots, N^n_{\sigma(n)})$ for any permutation $\sigma$ of $\{1, \dots, n\}$. The function $\varphi^n$ models the mutual interaction among particles.
 As before, we define the compensator and compensated point process associated with $N^n_i$ as follows
 \beqlb\label{eqn.210}
 \Lambda^n_i(t) = \Lambda^n_0(t) := \big\|\lambda^n_0\big\|_{L^1_t}
 \quad \mbox{and}\quad 
 M^n_i(t) := N^n_i(t) - \Lambda^n_0(t),\quad t\geq 0, \, 1\leq i\leq n. 
 \eeqlb 
In this section, we are interested in the long-run behavior of two empirical measures derived from the particle system, defined as
  \beqnn
  P^{(n)}_N(t,dx):= \frac{1}{n}\sum_{i=1}^n \delta_{N^{(n)}_i(t)}(dx) 
  \quad \mbox{and}\quad 
  P^{(n)}_M(t,dx):= \frac{1}{n}\sum_{i=1}^n \delta_{M^{(n)}_i(t)}(dx) ,\quad t\geq 0,
  \eeqnn
which are two c\`{a}dl\`{a}g measure-valued processes, where $\delta_x$ denotes the Dirac measure at point $x$. As we will see, the behavior of $P^{(n)}_N$ and $P^{(n)}_M$ is closely related to the long-term behavior of the \textsl{total size} of the particle system, defined by
 \beqlb\label{eqn.209}
 \overline{N^n}(t) := \sum_{i=1}^n N^n_i(t) ,\quad t\geq 0.
 \eeqlb 
We can identify $\overline{N^n}$ as a univariate Hawkes process with background rate $\mu_0^n$ and kernel $\varphi^n$. For convenience, we continue to denote by $\psi^n$ the resolvent of the second kind associated with $\varphi^n$, i.e., \beqlb\label{eqn.201}
 \psi^n(t) =  \sum_{k\geq 1} (\varphi^n)^{*k}(t) 
 \quad \mbox{and} \quad
 \psi^n(t) =\varphi^n(t) + \varphi^n*\psi^n(t),
 \quad t\geq 0.
 \eeqlb
Note that the intensity $\overline{\lambda^n}(t)$ of $\overline{N^n}$, the compensator $\overline{\Lambda^n}$ of $\overline{N^n}$, and the compensated martingale $\overline{M^n}(t)$ are all related to the underlying particle system, and we have for $t \geq 0$,
 \beqlb
 \overline{\lambda^n}(t) \ar:=\ar \sum_{i=1}^n \lambda^n_i(t) = n\cdot \lambda^n_0(t), \label{eqn.211.a}\\
 \overline{\Lambda^n} (t)\ar:=\ar \int_0^t \overline{\lambda^n}(s)ds  =n\cdot \Lambda^n_0(t), \label{eqn.211.b}\\
 \overline{M^n}(t) \ar:=\ar \overline{N}^n(t) - \overline{\Lambda^n}(t)= \sum_{i=1}^n  M^n_i(t) .\label{eqn.211.c}
 \eeqlb 
 
Following the methodology in Section~\ref{Sec.MainResultSL}, we consider a sequence of scaling parameters $\{\beta_n\}_{n \geq 1}$ satisfying Condition~\ref{Main.Condition.01} and define the spatially-scaled processes as
 \beqlb
 \overline{\Lambda^{(n)}}(t):=\beta_n^2\cdot \overline{\Lambda^n}(t),\quad 
 \overline{N^{(n)}}(t):=\beta_n^2\cdot \overline{N^n}(t),\quad 
 \overline{M^{(n)}}(t):= \beta_n \cdot  \overline{M^n}(t).
 \eeqlb
We also rescale each particle as follows
\beqlb
 \Lambda^{(n)}_i(t):= n \beta_n^2\cdot  \Lambda^n_i(t),\quad 
 N^{(n)}_i(t):=n \beta_n^2\cdot  N^n_i(t),\quad 
 M^{(n)}_i(t):=\sqrt{n} \beta_n \cdot  M^n_i(t), 
 \quad 1\leq i\leq n .
 \eeqlb  
 
We now state a tightness result for the finite-dimensional distributions associated with this mean-field system, which can be seen as an analogue of Theorem~\ref{Thm.MainResultSL.01} in this framework.

 \begin{theorem} \label{Thm.MainResultSL.03}
Assume that Condition~\ref{Main.Condition.01} with $d = 1$ holds and $\beta_n \cdot \mu_0^n \to a > 0$. For each $K \geq 1$, the sequence $\big\{ \big(\overline{\Lambda^{(n)}}, \overline{N^{(n)}}, \overline{M^{(n)}}, N^{(n)}_1, M^{(n)}_1, \cdots, N^{(n)}_K, M^{(n)}_K \big) \big\}_{n \geq 1}$ is tight. Moreover, suppose that $(\overline{U}, \overline{X}, \overline{Z}, X_1, Z_1, \cdots, X_K, Z_K )$ is an accumulation point along a subsequence $\{ n_k \}_{k \geq 1}$. Then Theorem~\ref{Thm.MainResultSL.01}(2) with $d = 1$ holds for $(\overline{U}, \overline{X}, \overline{Z})$. Denoting $\zeta = \lim\limits_{k \to \infty} n_k \beta_{n_k}^2 \in [0, \infty]$, we also have the following:
 	\begin{enumerate}
 		\item[(1)] If $\zeta = 0$, there exist mutually independent Brownian motions $W_1, \cdots, W_K$ that are independent of $(\overline{X}, \overline{Z}, B)$ such that  
 		\beqlb\label{eqn.216b}
 		X_i  =  \overline{X} 
 		\quad \mbox{and}\quad 	Z_i  \ar=\ar  W_i\circ\overline{X},\quad 1\leq i\leq K.
 		\eeqlb
 		
 		\item[(2)] If $\zeta \in (0,\infty)$, there exist 
 		mutually independent Poisson processes $N^\circ_1,\cdots,N^\circ_K$ with rate $1$ that are independent of $(\overline{X},\overline{Z},B)$ such that 
 		\beqlb \label{eqn.216c}
 		X_i  = \zeta\cdot N^\circ_i\circ \frac{\overline{X}}{\zeta}  
 		\quad \mbox{and}\quad 	Z_i \ar=\ar \sqrt{\zeta}\cdot \widetilde{N}^\circ_i \circ \frac{\overline{X}}{\zeta} ,\quad 1\leq i\leq K,
 		\eeqlb
 		where $\widetilde{N}^\circ_i(t):= N^\circ_i(t)-t$ is the compensated point process of $ N^\circ_i$. 
 		
 		\item[(2)] If $\zeta =\infty$, then $X_i\overset{\rm a.s.}= 0$ and  $Z_i\overset{\rm a.s.}= 0$ for $1\leq i\leq K$. 
 	\end{enumerate}
 \end{theorem}
  
The following result, proved in Section~\ref{Sec.Proof03}, describes the asymptotic behavior of $P^{(n)}_N$ and $P^{(n)}_M$.

 \begin{theorem}\label{Main.Thm03}
Under the assumptions of Theorem~\ref{Thm.MainResultSL.03}, we have 
\beqlb\label{eqn.213}
 (P^{(n_k)}_N, P^{(n_k)}_M)\to (P_X,P_Z),
 \eeqlb
weakly in $\mathrm{D}(\mathbb{R}_+; \mathrm{M}(\mathbb{R}_+) \times \mathrm{M}(\mathbb{R}))$ as $k \to \infty$, where the limit process $(P_X, P_Z)$ is given as follows:
 	\begin{enumerate}
 		\item[(1)] If $\zeta=0$, then $(P_X,P_Z)= (\delta_{\overline{X}},\mathcal{N}\circ\overline{X})$,  where $\mathcal{N} \in \mathrm{C}(\mathbb{R}_+;\mathrm{M}(\mathbb{R}))$  such that $\mathcal{N}(x)$ is a Gaussian distribution with mean $0$ and variance $x$.
 		
 		\item[(2)] If $\zeta\in(0,\infty)$, then $(P_X,P_Z)= \mathcal{P}_\zeta \circ \frac{\overline{X}}{\zeta}$, where $\mathcal{P}_\zeta  \in \mathrm{C}(\mathbb{R}_+;\mathrm{M}(\mathbb{R}_+\times \mathbb{R}))$ with  $\mathcal{P}_\zeta(x)$ being the joint distribution law of $\big(\zeta \cdot N_1^\circ(x),\sqrt{\zeta} \cdot \widetilde{N}_1^\circ(x)\big)$.

 		\item[(3)] If $\zeta=\infty$, then $(P_X,P_Z)= \delta_{(0,0)}$ is the Dirac measure at point $(0,0)$.  
 	\end{enumerate}
 	
\end{theorem}

\begin{remark}
Theorems~\ref{Thm.MainResultSL.03} and \ref{Main.Thm03} identify three distinct regimes depending on the asymptotic behavior of $n \beta_n^2$, ranging from complete synchronization when $n\beta_n^2 \to 0$, conditional independence when $n\beta_n^2 \to \zeta \in (0, \infty)$, to complete extinction when $n\beta_n^2 \to \infty$. These regimes can be understood intuitively by analyzing the rescaled processes.
More precisely, recall that for each $1 \leq i \leq n$, we have $N_i^n(t) = \Lambda^n_i(t) + M_i^n(t)$, which, after rescaling, turns to be
\begin{equation}
\label{eq:rescaledprocesses}
N_i^{(n)}(t) = \Lambda^{(n)}_i(t) + \sqrt{n \beta_n^2} \cdot M_i^{(n)}(t).
\end{equation}
In view of the proof of Theorem~\ref{Thm.MainResultSL.03}, the three processes $N_i^{(n)} $, $\Lambda^{(n)}_i$ and $M_i^{(n)}$ converge in distribution to $X_i$, $U_i$, and $Z_i$, respectively. 
 Note that all particles share the same intensity, i.e. $U_i = U$ for every $i$, and $M_i^{(n)}(t)$ is a martingale with quadratic variation $N_i^{(n)}(t)$. Similarly as in Theorem~\ref{Thm.MainResultSL.01}, we have
 \beqnn
 Z_i = B_i \circ X_i, \quad i\geq 1,
 \eeqnn
 where $B_1,B_2,\cdots$, are independent martingales. Substituting the limits of the processes into \eqref{eq:rescaledprocesses}, we obtain that 
 \beqnn
 X_i \approx U_i + \sqrt{n \beta_n^2} \cdot B_i \circ X_i.
 \eeqnn
 Consequently, the behavior of $X_i$ depends on the  limiting value of $n \beta_n^2$:
 \begin{itemize}
    \item When $n \beta_n^2 \to 0$, the martingale term $\sqrt{n \beta_n^2} \cdot M_i^{(n)}(t)$ becomes asymptotically negligible. This simplifies the approximation to
    \beqnn
    X_i \approx U,
    \eeqnn
    meaning that all particles behave asymptotically the same and are effectively indistinguishable.
    
    \item  In this intermediate regime when $n \beta_n^2 \to \zeta \in (0, \infty)$, we have
    \beqnn
    X_i \approx U_i + \sqrt{\zeta} \cdot B_i \circ X_i,
    \eeqnn
    which indicates that all particles remain conditionally independent, with the martingale term contributing to their individual randomness.

    \item When $n \beta_n^2 \to \infty$, the term $\sqrt{n \beta_n^2} \cdot B_i \circ X_i$ would diverge unless $B_i \circ X_i = 0$. Therefore, we conclude that
    \beqnn
    X_i = Z_i = U_i = 0,
    \eeqnn
    meaning that the particles cease to evolve and all processes converge to zero.
\end{itemize}
\end{remark}

\begin{remark}
As explained in the Introduction, our results can be compared with those in \cite{DelattreFournierHoffmann2016}, where $\phi^n = \phi$ is chosen independently of $n$. When $\|\phi\|_{L^1} < 1$, they proved a propagation of chaos as $n \to \infty$, where each particle becomes asymptotically independent and behaves as a simple inhomogeneous Poisson process, whose intensity is the solution of the Volterra integral equation
\begin{equation}
\label{eq:volterraintegralMF}
\overline{\lambda}_t = \mu + \int_0^t \phi(t-s) \overline{\lambda}_s \, ds.
\end{equation}
This result is similar to the second case in Theorems~\ref{Thm.MainResultSL.03} and \ref{Main.Thm03}, where all particle tend to be asymptotically independent and behave as a sequence of independent inhomogeneous Poisson processes. The main difference in that case is that \eqref{eq:volterraintegralMF} is replaced by a stochastic Volterra equation. This is intuitive and stems from the fact that nearly unstable Hawkes processes converge to a stochastic limit, whereas under stability assumptions, the law of large numbers ensures convergence toward a deterministic level.
\end{remark}

 \section{Proof of Theorem~\ref{Thm.MainResultSL.01}}
 \label{Sec.Proof01} 
 \setcounter{equation}{0}
 
In this section, we present the detailed proof of Theorem~\ref{Thm.MainResultSL.01}. The proof is divided into three parts:
 \begin{itemize}
    \item We first establish several uniform moment estimates for the Hawkes processes, which play an important role in the tightness argument; see Section~\ref{Sec.Moment}. 	
    
  \item The $C$-tightness of the sequence $\big\{ \big( \Lambda^{(n)}, N^{(n)}, M^{(n)} \big) \big\}_{n \geq 1}$ is proved in Section~\ref{Subsection.Tightness} using the Kolmogorov-Chentsov theorem, which states that stochastic processes $\{\xi^n\}_{n\geq 1} \subset \mathrm{D}(\mathbb{R}_+; \mathbb{R}^d)$ is $C$-tight if for each $T \geq 0$, there exist constants $C, p, \theta > 0$ such that for any $h \in (0, 1)$,
  	\beqnn
  	\sup_{n\geq 1}\sup_{t\in[0,T]}\mathbf{E}\Big[ \big|\xi^{n}(t+h)-\xi^{n}(t)\big|^p\Big]\leq C\cdot h^{1+\theta}. 
  	\eeqnn

    \item The characterization of accumulation points is provided in Section~\ref{Subsection.Characterization}.
     \end{itemize}

 \subsection{A priori estimates} \label{Sec.Moment}
 
 We begin by providing a uniform upper bound estimate on the means $\{ \mathbf{E}[\beta_n^2 \cdot \lambda^n] \}_{n \geq 1}$. Taking expectations on both sides of \eqref{MartingaleRep.n}, we obtain
 \beqnn
 \mathbf{E}\big[\lambda^n(t)\big] =\mu^n + \big\|\psi^n\big\|_{L^1_t} \cdot \mu^n.
 \eeqnn
Using Condition~\ref{Main.Condition.01} and the assumption that $\beta_n \mu^n \to a$, for each $T \geq 0$, there exists a constant $C > 0$ depending only on $T$ such that for any $n \geq 1$,
 \beqlb\label{eqn.301} 
 \sup_{t\in[0,T]}\mathbf{E}\Big[\beta_n^2\cdot \lambda^n(t)\Big] = \beta_n^2\mu^n + \beta_n\big\|\psi^n\big\|_{L^1_T} \cdot \beta_n\mu^n \leq C
 \quad \mbox{and}\quad 
 \mathbf{E}\big[  \Lambda^{(n)}(T)\big]\leq C\cdot T. 
 \eeqlb
We now prove in the next lemma that the same kind of $L^2$-estimates hold.
 
 \begin{lemma}\label{Lemma.301}
 	For each $T\geq 0$, there exists a constant $C>0$ such that for any $n\geq 1$,
 	\beqlb\label{eqn.302}
    \sup_{t\in[0,T]}\mathbf{E}\Big[\big|  \beta_n^2 \cdot \lambda^{n}(t)\big|^2 \Big]+   \mathbf{E}\bigg[ \sup_{t\in[0,T]}\big|   M^{(n)}(t)\big|^2 \bigg] \leq C.
 	\eeqlb
 \end{lemma}

\begin{proof}
 Using the Burkholder-Davis-Gundy inequality and Equation \eqref{eqn.301}, we have
 \beqnn
 \mathbf{E}\bigg[\sup_{t\in[0,T]}   \big|    M^{(n)}(t)\big|^2 \bigg] 
 \leq C \cdot \mathbf{E}\Big[   N^{(n)}(T)  \Big] = C \cdot \mathbf{E}\big[  \Lambda^{(n)}(T)\big] 
 \leq C\cdot  T 
 \eeqnn
 for some constant $C>0$ independent of $n$. 
 Moreover, using Equations \eqref{MartingaleRep.n} and \eqref{eqn.301}, we also have
 \beqnn
 \mathbf{E}\Big[\big| \beta_n^2\lambda^n(t)\big|^2\Big]
  \ar\leq\ar 4 \cdot \Big| \beta_n^2\mu^n+ \beta_n\big\|\psi^n\big\|_{L^1_t} \cdot  \beta_n\mu^n\Big|^2 + 4\cdot \mathbf{E}\bigg[\Big| \beta_n^2\int_0^t \psi^n(t-s)dM^n(s)\Big|^2\bigg] \cr
  \ar\leq \ar C + 4\cdot \mathbf{E}\bigg[\Big| \beta_n^2\int_0^t \psi^n(t-s)dM^n(s)\Big|^2\bigg].
 \eeqnn
 Similarly as in the proof of Lemma 3.9 in \cite{HorstXuZhang2023a}, we apply
 the Burkholder-Davis-Gundy inequality to the last expectation and we get 
 \beqnn
 \mathbf{E}\bigg[\Big| \beta_n^2\int_0^t \psi^n(t-s)dM^n(s)\Big|^2\bigg] 
 \ar\leq\ar C \cdot \mathbf{E}\bigg[  \beta_n^4\int_0^t |\psi^n(t-s)|^2dN^n(s) \bigg] \cr
 \ar=\ar C \cdot   \beta_n^2\int_0^t |\psi^n(t-s)|^2\mathbf{E}\big[ \beta_n^2\lambda^n(s)\big]ds , 
 \eeqnn
 which is bounded uniformly in $n\geq 1$ and $t\in[0,T]$ because the first inequality in \eqref{eqn.301} and the fact that $ \beta_n^2\int_0^T |\psi^n(s)|^2\, ds$ is uniformly bounded; see Condition~\ref{Main.Condition.01}.
\end{proof}

 \begin{corollary} \label{Corollary.302}
 	For each $T\geq 0$, we have $ \sup_{t\in[0,T]}  \big|  N^{(n)}(t)- \Lambda^{(n)}(t)  \big|\overset{\rm p}\to 0$ as $n\to\infty$.
%
%
 \end{corollary}
 \proof Note that $ N^{(n)}(t)-  \Lambda^{(n)}(t)= \beta_n M^{(n)}(t)$. 
 Therefore, by \eqref{eqn.302}, we have 
 \beqlb\label{eqn.303}
 \mathbf{E}\bigg[ \sup_{t\in[0,T]} \big| N^{(n)}(t)-  \Lambda^{(n)}(t)\big|^2  \bigg]
 = \beta_n^2\cdot\mathbf{E}\bigg[\sup_{t\in[0,T]} \big|  M^{(n)}(t)\big|^2  \bigg] 
 \leq  C \cdot \beta_n^2 ,
 \eeqlb
 which goes to $0$ as $n\to\infty$.
 \qed

 \subsection{Theorem~\ref{Thm.MainResultSL.01}(1): $C$-tightness}\label{Subsection.Tightness}
 
 In this section, we prove the $C$-tightness of the sequence $\big\{\big( \Lambda^{(n)},   N^{(n)},  M^{(n)} \big)\big\}_{n\geq 1}$. 
 We first consider $\big\{  \Lambda^{(n)}\big\}_{n\geq 1} $. We plan to use the Kolmogorov-Chentsov criterion, which require a uniform bound on the increment of $\Lambda^{(n)}$.
 
 \begin{proposition}\label{Prop.303}
 For each $T\geq 0$, there exists a constant $C>0$ such that  for any $h\in (0,1)$, 
 \beqlb\label{eqn.304}
 \sup_{n\geq 1} \sup_{t\in[0,T]} \mathbf{E}\Big[ \big|  \Lambda^{(n)}(t+h)- \Lambda^{(n)} (t)\big|^2 \Big] 
 \ar\leq\ar C\cdot h^2. 
 \eeqlb
 	
 \end{proposition}
 \proof By the definition of $\Lambda^{n}$, we have
 \beqnn
 \Lambda^{(n)}(t+h)- \Lambda^{(n)}(t)= \int_{t}^{t+h} \beta_n^2\lambda^n(s)ds 
 \eeqnn
 and hence by H\"older's inequality,
 \beqnn 
 \big|  \Lambda^{(n)}(t+h)- \Lambda^{(n)}(t) \big|^2 \leq 
 h\cdot \int_{t}^{t+h} \big|\beta_n^2\lambda^n(s)\big|^2ds.
 \eeqnn
 Taking expectations on both sides of this inequality and then using Fubini's theorem along with \eqref{eqn.302}, 
 \beqlb\label{eqn.310}
 \mathbf{E}\Big[ \big|  \Lambda^{(n)}(t+h)- \Lambda^{(n)} (t)\big|^2 \Big] 
 \leq C\cdot h^2,
 \eeqlb
 for some constant $C>0$ depending only on $T$. 
 \qed 
 
 \begin{corollary}\label{Corollary.304}
  The sequence $\big\{\big( \Lambda^{(n)}, N^{(n)},  M^{(n)} \big)\big\}_{n\geq 1}$ is $C$-tight. 
 \end{corollary}
 \proof By Corollary 3.33(b) in \cite[p.353]{JacodShiryaev2003}, it suffices to prove that the three sequences $\big\{ \Lambda^{(n)}\big\}_{n\geq 1} $, $\big\{  N^{(n)}\big\}_{n\geq 1} $ and $\big\{  M^{(n)} \big\}_{n\geq 1}$ are $C$-tight separately. 
 Proposition~\ref{Prop.303} and the Kolmogorov-Chentsov theorem directly induce the $C$-tightness of $\{  \Lambda^{(n)}\}_{n\geq 1} $. Corollary~\ref{Corollary.302} therefore also yields the $C$-tightness of $\{  N^{(n)}\}_{n\geq 1} $. 
  Finally, by \eqref{eqn.211} the martingale $ M^{(n)}$ has predictable quadratic co-variation 
  \beqlb \label{eqn.306}
  \big\langle  M^{(n)}_i,  M^{(n)}_j \big\rangle = \mathbf{1}_{\{ i\neq j \}}\cdot    \Lambda^{(n)},\quad   1\leq i,j\leq d.
  \eeqlb
  By Theorem 4.13 in \cite[p.358]{JacodShiryaev2003}, the $C$-tightness of $\{  M^{(n)}\}_{n\geq 1} $ follows from the $C$-tightness of  $\{   \Lambda^{(n)}\}_{n\geq 1} $ and the fact that all jumps of $  M^{(n)}$ are uniformly bounded by $\beta_n$ that vanishes as $n\to\infty$.  
  \qed

 \subsection{Theorem~\ref{Thm.MainResultSL.01}(2): Characterization} \label{Subsection.Characterization}
 Assume that $(U,X,Z)$ is the limit process of a sub-sequence $\{( \Lambda^{(n_k)} , N^{(n_k)}, M^{(n_k)} )\}_{k\geq 1}$. 
 By Corollary~\ref{Corollary.302}(1) and the Skorokhod representation theorem, we have $U\overset{\rm a.s.}=X$ and may assume that 
 \beqlb\label{eqn.305}
 \big( \Lambda^{(n_k)} , N^{(n_k)}, M^{(n_k)} \big) 
 \overset{\rm a.s.}\to \big(X,X,Z \big),
 \eeqlb
 in $\mathrm{D}(\mathbb{R}_+;\mathbb{R}_+^d\times\mathbb{R}_+^d\times \mathbb{R}^d)$ as $k\to\infty$. 
 Using Proposition~1.17(b) in \cite[p.328]{JacodShiryaev2003} along with the continuity of $(X,Z)$, the convergence also holds for the uniform norm on compact subsets. More precisely, we have for any $T\geq 0$, 
 \beqlb\label{eqn.307}
 \sup_{t\in [0,T]} \big| \Lambda^{(n_k)} (t) - X(t) \big|+\sup_{t\in [0,T]} \big| N^{(n_k)} (t) - X(t) \big| +  \sup_{t\in [0,T]} \big| M^{(n_k)} (t) - Z(t) \big| \overset{\rm a.s.}\to 0,
 \eeqlb
 as $n\to\infty$. This limit and the monotonicity of $N^{(n_k)}$ imply that $X$ is non-decreasing. 
 
 We now prove that $(X,Z)$ is a weak solution of \eqref{eqn.LimitSVE}. 
 By Theorem~6.26 in \cite[p.384]{JacodShiryaev2003} and \eqref{eqn.306}-\eqref{eqn.305}, the limit process $Z$ is a continuous martingale with predictable quadratic covariation given by
 \beqnn
 \big\langle Z_i,Z_j \big\rangle =  \mathbf{1}_{\{i=j\}}\cdot X_i,\quad  1\leq i,j\leq d.
 \eeqnn 
 By the martingale representation theorem; see Theorem~7.1 in \cite[p.84]{IkedaWatanabe1989}, there exists an extension of the original probability space on which is defined a standard $d$-dimensional Brownian motion $B$ such that 
 \beqnn
 Z_i(t) = B_i\circ X_i(t)  ,\quad  t\geq 0,\, 1\leq i\leq d,
 \eeqnn
 and hence the first equality in \eqref{eqn.LimitSVE} holds. 
 On the other hand, integrating both sides of \eqref{MartingaleRep.n} over the interval $[0,t]$ we have 
 \beqnn
 \Lambda^{n_k}(t)=\mu^{n_k}\cdot t + \int_0^t ds \int_0^s\psi^{n_k}(r)dr \cdot \mu^{n_k} + \int_0^t ds \int_0^s \psi^{n_k}(s-r)dM^{n_k}(r) .
 \eeqnn
 Applying Fubini's theorem and the stochastic Fubini theorem to the last two terms on the right-hand side of this equality and then multiplying both sides by $\beta_n^2$, we get
 \beqlb\label{eqn.308}
 \Lambda^{(n_k)}(t)\ar=\ar \beta_{n_k}^2\mu^{n_k}\cdot t + \int_0^t ds \int_0^s\beta_{n_k} \psi^{n_k}(r)dr \cdot \beta_{n_k}\mu^{n_k}   + \int_0^t  \beta_{n_k}\psi^{n_k}(s)
 \cdot  M^{(n_k)}(t-s)ds\cr
 \ar=\ar \beta_{n_k}^2\mu^{n_k}\cdot t + \int_0^t ds \int_0^s\beta_{n_k} \psi^{n_k}(r)dr \cdot \beta_{n_k}\mu^{n_k}  + \int_0^t  \beta_{n_k}\psi^{n_k}(s)
 \cdot  Z(t-s)ds\cr
 \ar\ar  + \int_0^t  \beta_{n_k}\psi^{n_k}(s)
 \cdot \Big( M^{(n_k)}(t-s)-Z(t-s)\Big)ds.
 \eeqlb
 Using \eqref{eqn.307} and then Condition~\ref{Main.Condition.01}, we have 
 \beqnn
 \sup_{t\in[0,T]}\bigg|  \int_0^t  \beta_{n_k}\psi^{n_k}(s)
 \cdot \Big( M^{(n_k)}(t-s)-Z(t-s)\Big)ds \bigg| 
 \ar\leq\ar  \int_0^T  \beta_{n_k}\psi^{n_k}(s)ds
 \cdot \sup_{t\in[0,T]}\Big| M^{(n_k)}(t)-Z(t)\Big| ,
 \eeqnn 
 which vanishes almost surely as $k\to\infty$. 
 Additionally, for each $t\geq 0$, by \eqref{eqn.212} we also have
 \beqnn
 \int_0^t  \beta_{n_k}\psi^{n_k}(s)
 \cdot  Z(t-s)ds \overset{\rm a.s.}\to \int_0^t  f (s)
  Z(t-s)ds ,
 \eeqnn
 as $k\to\infty$.
 From with these limits, \eqref{eqn.307}, Condition~\ref{Main.Condition.01} and the fact that $\beta_{n_k}\mu^{n_k} \to a$, we obtain the second equality in \eqref{eqn.LimitSVE} by passing each term in \eqref{eqn.308} to the limit.

 \section{Proof of Theorem~\ref{Thm.MainResultSL.03}}
 \label{Sec.Proof02} 
 \setcounter{equation}{0}

 First note that $\{ (\mu^n_0,\psi^n) \}_{n\geq 1}$ satisfies assumptions in Theorem~\ref{Thm.MainResultSL.01} and therefore, the sequence $\big\{\big(   \overline{\Lambda^{(n)}} ,  \overline{N^{(n)}}, \overline{M^{(n)}} \big) \big\}_{n\geq 1}$ is $C$-tight. By Corollary 3.33(b) in \cite[p.353]{JacodShiryaev2003}, it remains to prove that the sequence 
\beqlb\label{eqn.417}
\big\{ (  N^{(n)}_1,  M^{(n)}_1, \cdots,   N^{(n)}_K,  M^{(n)}_K) \}_{n\geq 1}
\eeqlb
is tight and Theorem~\ref{Thm.MainResultSL.03}(1)-(3) holds for every limit process.
First note that it is enough to prove Theorem~\ref{Thm.MainResultSL.03} under the additional assumption
 \beqlb\label{eqn.420}
 n\beta_n^2\to \zeta\in[0,\infty],\quad \mbox{as }n\to\infty.  
 \eeqlb
Indeed, once Theorem~\ref{Thm.MainResultSL.03} is proved under \eqref{eqn.420}, it can also be proved in the general case using the following argument.
 \begin{enumerate}
 	\item[(1)] By Prohorov's theorem the sequence \eqref{eqn.417} is tight if and only if it is relatively compact, i.e., every subsequence contains a weak convergent subsequence. Then, for each subsequence $\{n_k \}_{k\geq 1}$, one can always find a sequence $\{n'_k \}_{k\geq 1} \subset \{n_k \}_{k\geq 1}$ such that \eqref{eqn.420} holds. Then, the subsequence of \eqref{eqn.417}  along $\{n'_k \}_{k\geq 1}$ is tight and also relatively compact. This implies that the sequence \eqref{eqn.417} is tight.
 	
 	\item[(2)] Assume that $(\overline{U},\overline{X}, X_1,Z_1, \cdots, X_K, Z_K)$ is a limit process along some subsequence $\{n_k \}_{k\geq 1}$ and assume that the limit $n_k\beta_{n_k}^2\to \zeta\in[0,\infty]$ fails as $k\to\infty$. Then one can always find two different subsequences $\{n'_k \}_{n\geq 1}, \{n^{''}_k \}_{n\geq 1}\subset\{n_k \}_{k\geq 1}$ and two different constants $\zeta_1,\zeta_2\in[0,\infty]$ such that $n'_k \beta_{n'_k}^2 \to \zeta_1$ and $n''_k \beta_{n''_k}^2 \to \zeta_2$ as $k\to\infty$. 
 	Using the preceding results again and our assumption, both the two subsequences of $\big\{\big(   \overline{\Lambda^{(n)}} ,  \overline{N^{(n)}}, \overline{M^{(n)}}, N^{(n)}_1,  M^{(n)}_1, \cdots,   N^{(n)}_K,  M^{(n)}_K \big) \big\}_{n\geq 1}$  along $\{n'_k \}_{k\geq 1}$ and $\{n^{''}_k \}_{n\geq 1}$ converge weakly to the same limit $(\overline{U},\overline{X}, X_1,Z_1, \cdots, X_K, Z_K)$. 
 	By Theorem~\ref{Thm.MainResultSL.03}(1)-(3), the process $(X_1,Z_1, \cdots, X_K, Z_K)$ varies for different $\zeta$ and hence $\zeta_1=\zeta_2$. 
 	\end{enumerate}

We proceed to prove Theorem~\ref{Thm.MainResultSL.03} separately with the cases $\zeta = 0$, $0 < \zeta < \infty$, and $\zeta = \infty$ in Sections~\ref{sec:proof2:I}, \ref{sec:proof2:II} and \ref{sec:proof2:III} respectively.
 As a preparation, we first give some auxiliary moment estimates. The results in Section~\ref{Sec.Proof01} with $d=1$ still hold for  $\big\{\big( \overline{\lambda^{(n)}} ,  \overline{\Lambda^{(n)}} ,  \overline{N^{(n)}}, \overline{M^{(n)}} \big) \big\}_{n\geq 1}$. In particular, for each $T\geq 0$, a direct consequence of Lemma~\ref{Lemma.301} and \eqref{eqn.211.a} shows that there exists a constant $C>0$ such that for any $p\leq 2$ and  $n\geq 1$,
 \beqlb\label{eqn.401}
 \sup_{t\in[0,T]}\mathbf{E}\Big[\big|  \lambda^{(n)}_0(t)\big|^p \Big] 
 =\sup_{t\in[0,T]}\mathbf{E}\Big[\big|   \overline{\lambda^{(n)}}(t)\big|^p \Big] \leq C
 \quad \mbox{and}\quad 
 \mathbf{E}\Big[\big|  \Lambda^{(n)}_0(T)\big|^p \Big] 
 = \mathbf{E}\Big[\big|   \overline{\Lambda^{(n)}}(T)\big|^p \Big] \leq C.
 \eeqlb

 \subsection{Case~I: $n\beta_n^2\to\zeta=0$} 
 \label{sec:proof2:I}
 
The proof is carried out in two steps: we first prove the $C$-tightness of the processes and then we characterize the limit.
 
 \textit{\textbf{$C$-tightness.}} 
 By  Corollary 3.33 in \cite[p.353]{JacodShiryaev2003}, it suffices to prove that for each $1\leq i\leq K$, the sequence $\big\{  (\Lambda^{(n)}_0,  N^{(n)}_i,   M^{(n)}_i)\big\}_{n\geq K}$ is $C$-tight.
 By Proposition~\ref{Prop.303} and \eqref{eqn.211.a}-\eqref{eqn.211.b}, for each $T\geq 0$, there exists a constant $C>0$ such that  for any $h\in (0,1)$ and $n\geq 1$, 
 \beqnn
 \sup_{t\in[0,T]} \mathbf{E}\Big[ \big| \Lambda^{(n)}_0(t+h)-  \Lambda^{(n)}_0(t)\big|^2 \Big] 
 =\sup_{t\in[0,T]} \mathbf{E}\Big[ \big|  \overline{\Lambda^{(n)}}(t+h)-\overline{\Lambda^{(n)}}(t)\big|^2 \Big] 
 \leq C\cdot h^2. 
 \eeqnn 
 Recall also that  $ \sqrt{n}\beta_n\cdot M^{(n)}_i = N^{(n)}_i- \Lambda_0^{(n)}$ and $\sqrt{n}\beta_n\to0$ as $n\to\infty$. 
 By repeating the proofs of Corollary~\ref{Corollary.302}(1) and \ref{Corollary.304}, we also can  prove that as $n\to\infty$,
 \beqlb \label{eqn.418}
 \sup_{t\in[0,T]}\big|  N^{(n)}_i(t)- \Lambda_0^{(n)}(t)\big| \overset{\rm p}\to 0
 \eeqlb 
 and hence the sequence  $\big\{( \Lambda^{(n)}_0,   N^{(n)}_i, M^{(n)}_i)\big\}_{n\geq K}$ is $C$-tight for any $i=1,\cdots,K$.
 
 \textit{\textbf{Characterization.}} 
 Assume that $(\overline{X}, \overline{Z}, X_1,Z_1, \cdots, X_K, Z_K) \in \mathrm{C}\big(\mathbb{R}_+;(\mathbb{R}_+\times \mathbb{R})^{1+K}\big)$ is a limit process along a sub-sequence $\{ n_k \}_{k\geq 1}$. 
 By Theorem~\ref{Thm.MainResultSL.01}(2) with $d=1$, the process $(\overline{X}, \overline{Z})$ is a weak solution to \eqref{eqn.LimitSVE}. 
 It suffices to prove that $\big(X_i, Z_i\big)$ satisfies \eqref{eqn.216b} and $\big\langle B,W_i\big\rangle\equiv 0$ a.s. 
 
 By \eqref{eqn.211.b} and the fact that $ \overline{\Lambda^{(n_k)}}  \to \overline{X}$ weakly in $\mathrm{D}(\mathbb{R}_+;\mathbb{R}_+)$, we have $ \Lambda_0^{(n_k)} \to \overline{X}$ weakly 
 in $\mathrm{D}(\mathbb{R}_+;\mathbb{R}_+)$ as $k\to\infty$. 
 This along with \eqref{eqn.418} implies that 
 \beqnn
 X_i\overset{\rm a.s.}=\overline{X}.
 \eeqnn
 On the other hand, for $1\leq i,j \leq K$ and $k\geq 1$ fixed, we see using \eqref{eqn.211} that the two martingales $M^{(n_k)}_i$ and $M^{(n_k)}_j$ have quadratic co-variation
 \beqlb\label{eqn.415}
 \big[ M^{(n_k)}_i, M^{(n_k)}_j \big] = \mathbf{1}_{\{ i=j \}}\cdot  N^{(n_k)}_i 
  \to \mathbf{1}_{\{ i=j \}}\cdot \overline{X} , 
 \eeqlb
 weakly in $\mathrm{D}(\mathbb{R}_+;\mathbb{R}_+)$ as $k\to\infty$. 
 By Theorem~6.26 in \cite[p.384]{JacodShiryaev2003}, this implies that the limit process $(Z_1,\cdots, Z_K)$ is a continuous $K$-dimensional martingale with predictable quadratic co-variation 
 \beqnn
 \big\langle Z_i, Z_j \big\rangle = \mathbf{1}_{\{ i=j \}}\cdot \overline{X},\quad 1\leq i,j \leq K.
 \eeqnn 
 By the martingale representation theorem; see Theorem~7.1 in \cite[p.84]{IkedaWatanabe1989}, there exists an extension of the original probability space on which is defined a $K$-dimensional Brownian motion $W=(W_1 ,\cdots, W_K)$ such that 
  \beqlb\label{eqn.416}
  Z_i =W_i\circ \overline{X},\quad 1\leq i\leq K.
  \eeqlb
  
 It remains to prove that $\langle B,W_i\rangle\equiv 0$. 
 Using Theorem~6.26 in \cite[p.384]{JacodShiryaev2003} we have 
 \beqnn
 \big[  \overline{M^{(n_k)}},  M^{(n_k)}_i \big]
 \to 
 \big[ \overline{Z},Z_i\big]
 =\big\langle \overline{Z},Z_i\big\rangle,
 \eeqnn
 weakly in $\mathrm{D}(\mathbb{R}_+;\mathbb{R})$ as $k\to\infty$. 
 Since $\big[  M^{n_k}_i,M^{n_k}_j \big]=0$ if $i\neq j$.
 by \eqref{eqn.211.c} and \eqref{eqn.415}  we have  
 \beqnn
 \big[  \overline{M^{(n_k)}},  M^{(n_k)}_i \big] = \frac{1}{\sqrt{n_k} }\cdot  \big[ M^{(n_k)}_i  ,  M^{(n_k)}_i \big],
 \eeqnn
 which converges to $0$ weakly in $\mathrm{D}(\mathbb{R}_+;\mathbb{R})$ and hence $\big\langle \overline{Z},Z_i\big\rangle \equiv 0$ a.s. 
 On the other hand, by \eqref{eqn.216b} and \eqref{eqn.416} we also have 
 \beqnn
 \big\langle \overline{Z},Z_i\big\rangle
 = \big\langle B\circ\overline{X},W_i\circ\overline{X}\big\rangle =\big\langle B,W_i\big\rangle\circ\overline{X} = 0,
 \eeqnn
 which immediately yields $\big\langle B,W_i\big\rangle \equiv 0$ a.s.

 \subsection{Case~II:  $n\beta_n^2\to\zeta\in(0,\infty)$} 
  \label{sec:proof2:II}

This case is more intricate and relies heavily on an auxiliary process introduced to prove conditional independence of the laws of $(X_1,Z_1), \dots, (X_L, Z_L)$. To define this coupling rigorously, we first give a Poisson representation of Hawkes process. 
 
\textit{\textbf{Poisson representation.}}
Applying Theorem~7.4 in \cite[p.90]{IkedaWatanabe1989}, we can suppose up to a probability space enlargement that for each $n\geq K$, there exist some independent Poisson measures $\{ \Pi^n_i(ds,dz) \}_{1\leq i\leq n}$ on $\mathbb{R}_+^2$ with intensity $ds\,dz$ such that 
 \beqlb\label{eqn.410.0}
 N^n_i(t)= \int_0^t \int_0^{\lambda^n_0(s-)} \Pi^n_i(ds,dz)
 \quad \mbox{and}\quad 
 M^n_i(t)= \int_0^t \int_0^{\lambda^n_0(s-)} \widetilde{\Pi}^n_i(ds,dz),\quad t\geq 0,\, 1\leq i\leq n,
 \eeqlb
 where $\widetilde{\Pi}^n_i(ds,dz):=\Pi^n_i(ds,dz)-dsdz$ and where the common intensity  $\lambda^n_0$ is given by 
 \beqlb\label{eqn.410}
 \lambda^n_0(t)=\frac{\mu_0^n}{n}  +  \frac{1}{n}\sum_{j=1}^n \int_0^t\int_0^{\lambda^n_0(s-)} \varphi^n(t-s) \Pi^n_j(ds,dz), \quad t\geq 0 .
 \eeqlb

 \textit{\textbf{Definition of the auxiliary processes.}}
 For each $n \geq K$, we define a $n$-dimensional random point process $\mathcal{N}^n:= (\mathcal{N}^n_1,\cdots,\mathcal{N}^n_n)$ by
 \beqlb\label{eqn.411.0}
 \mathcal{N}^n_i(t) := \int_0^t \int_0^{\theta^n(s-)}\Pi^n_i(ds,dz),\quad t\geq 0,\, 1\leq i\leq n,
 \eeqlb
 where the non-negative process $\theta^n$ is given by
 \beqlb\label{eqn.411}
 \theta^n(t):= \frac{\mu_0^n}{n} + \frac{1}{n} \sum_{j=K+1}^n  \int_0^t \int_0^{\theta^n(s-)}\varphi^n(t-s)\Pi^n_j(ds,dz).
 \eeqlb
In other words, $\mathcal{N}^n$ is $n$-variate Hawkes process where all components have the common intensity $\theta^n$. With matrix notation, the background rate of $\mathcal{N}^n$ is $\mu_0^n/n$ and the kernel is given by
\begin{equation*}
\begin{pmatrix}
\mathbf{0}_{n, K} & \varphi^n \mathbf{1}_{n, n-K}
\end{pmatrix},
\end{equation*}
where $\mathbf{0}_{n, K}$ denotes the $n \times K$ matrix whose components are $0$ and $\mathbf{1}_{n, n-K}$ the the $n \times K$ matrix whose components are $1$. For $1\leq i\leq n$, we write 
 \beqnn
 \Theta^n(t):= \int_0^t \theta^n(s)ds 
 \quad \mbox{and}\quad
 \mathcal{M}^n_i(t):= \mathcal{N}^n_i(t)-\Theta^n(t)=  \int_0^t \int_0^{\theta^n(s-)}\widetilde{\Pi}^n_i(ds,dz), \quad t\geq 0,
 \eeqnn
 for the compensator and compensated point process of $\mathcal{N}^n_i$.  
 Furthermore, we write 
 \beqlb \label{eqn.413.a}
 \overline{\mathcal{N}^n}(t)\ar:=\ar \sum_{i=K+1}^n  \mathcal{N}^n_i(t) = \sum_{i=K+1}^n \int_0^t \int_0^{\theta^n(s-)}\Pi^n_i(ds,dz).
 \eeqlb
Note that unlike $\overline{N^{(n)}}$, the point process $\overline{\mathcal{N}^n}(t)$ is not (in general) a Hawkes process, due to the non-interchangeability of its component. 
However,  it is still an increasing point process and its compensator is given by
 \beqlb
 \overline{\Theta^n}(t)\ar:=\ar (n-K)\cdot \Theta^n(t)=(n-K)\cdot \int_0^t \theta^n(s)ds , \label{eqn.413.b}
  \eeqlb
 and the compensated point process associated to $\overline{\mathcal{N}^n}(t)$ is given by
 \beqlb
 \overline{\mathcal{M}^n}(t)\ar:=\ar\overline{\mathcal{N}^n}(t)-  \overline{\Theta^n}(t)= \sum_{i=K+1}^n  \int_0^t \int_0^{\theta^n(s-)}\widetilde{\Pi}^n_i(ds,dz).\label{eqn.413.c} 
 \eeqlb
 The next proposition is a direct consequence of  the mutual independence among the Poisson random measures $\{\Pi^n_i(ds,dz)\}_{i\geq 1}$. The detailed proof is omitted.
 
 \begin{proposition}\label{Prop.407}
 	For each $n\geq K$ and $1\leq i\leq K$, the process $(\theta^n, \Theta^n,\overline{\mathcal{N}^n},\overline{\Theta^n}, \overline{\mathcal{M}^n})$ is independent of $\Pi^n_i(ds,dz)$. 
 \end{proposition}
 
 This means that $\mathcal{N}^{n}_1,\cdots, \mathcal{N}^{n}_K$ can be seen as Cox processes indexed by the random measure generated by the non-decreasing process $\overline{\Theta^n}$. Moreover, conditioned on $(\theta^n, \Theta^n,\overline{\mathcal{N}^n},\overline{\Theta^n}, \overline{\mathcal{M}^n})$, the processes $\mathcal{N}^{n}_1,\cdots, \mathcal{N}^{n}_K$ are mutually independent. Therefore, we obtain the following. 
 
 \begin{corollary}\label{Coro.411}
 	There exist mutually independent Poisson processes $N_1^\circ,\cdots,N^\circ_K$ with rate $1$ that are independent of $\{(\theta^n, \Theta^n,\overline{\mathcal{N}^n},\overline{\Theta^n}, \overline{\mathcal{M}^n})\}_{n\geq 1}$ such that for any $n\geq K$,
 	\beqnn
 	\mathcal{N}^{n}_i \overset{\rm d}= N_i^\circ\circ \Theta^{n} 
 	\quad \mbox{and}\quad 
 	\mathcal{M}^{n}_i  \overset{\rm d}= \widetilde{N}^\circ_i\circ \Theta^{n} ,\quad 1\leq i\leq K,
 	\eeqnn
 	where $\widetilde{N}^\circ_i(t) := N^\circ_0(t) -t$. 
 \end{corollary}
 
The rest of the proof essentially consist in establishing the asymptotic behavior of $\Theta^n$, studying the asymptotic behavior of $\mathcal{N}^{n}_1,\cdots, \mathcal{N}^{n}_K$ and proving that the asymptotic behaviour of $\mathcal{N}^{n}_1,\cdots, \mathcal{N}^{n}_K$ and $N^{n}_1,\cdots, \mathcal{N}^{n}_K$ is the same. We start with a comparison result between the initial Hawkes process and our auxiliary process.
 
 \begin{proposition}\label{Prop.401}
 	For each $n\geq K$ and $1\leq i\leq n$,  we have almost surely
 	\beqnn
 	\lambda^n_0(t)\geq \theta^n(t), \quad 
 	N^n_i(t) \geq  \mathcal{N}^n_i(t)
 	\quad \mbox{and} \quad 
 	\overline{N^n}(t) \geq \overline{\mathcal{N}^n}(t),
 	\quad t\geq 0.
 	\eeqnn 
 \end{proposition}
 \proof Let $\tau_{0}=0$ and $\tau_{k}:= \inf\{t\geq \tau_{k-1}: N_i(t)-N_i(t-)= 1 \mbox{ for some } 1\leq i\leq K\}$ for $k\geq 1$.  
 Comparing the right-hand sides of \eqref{eqn.410} and \eqref{eqn.411}, we see that 
 \beqnn
 \lambda^n_0(t) \overset{\rm a.s.}= \theta^n(t)
 \quad \mbox{and}\quad 
 N^n_i(t)\overset{\rm a.s.}= \mathcal{N}^n_i(t),\quad t< \tau_1,\, 1\leq i \leq n.
 \eeqnn
We then define $\tau_+:= \inf\{t > 0:  \lambda^n_0(t)< \theta^n(t) \}$ with the convention that $\inf \emptyset = + \infty$. We already know that $\tau_+ \geq \tau_1 > 0$ a.s. and we want to prove that $\tau_+=\infty$. Suppose that $\tau_+ < \infty$. For $t < \tau_+$, we have $\theta^n(t) \leq \lambda^n_0(t)$ and by left-continuity of $\theta^n$ and $\lambda^n_0$, we also have $\theta^n(\tau_+) \leq \lambda^n_0(\tau_+)$. Therefore, by \eqref{eqn.410.0} and \eqref{eqn.411.0}, we know that the jump times of $\mathcal{N}^{n}_{1}, \dots, \mathcal{N}^{n}_n$ on $[0, \tau_+]$ are necessarily jump times of $N^{n}_1, \dots, N^{n}_K$ on $[0, \tau_+]$. For $t > \tau_+$, this implies that
\begin{align*}
\lambda^n_0(t) 
&= 
\frac{1}{n} \Big( \mu^n_0 + \sum_{j=1}^n \int_{[0,\tau]} \varphi^n(t-s) \,dN^n_j(s) + \sum_{j=1}^n \int_{(\tau,t]} \varphi^n(t-s) \,dN^n_j(s) \Big)
\\
&\geq
\frac{1}{n} \Big( \mu^n_0 + \sum_{j=K+1}^n \int_{[0,\tau]} \varphi^n(t-s) \,d\mathcal{N}^n_j(s) + \sum_{j=K+1}^n \int_{(\tau,t]} \varphi^n(t-s) \,dN^n_j(s) \Big)
\\
&\geq
\widetilde{\lambda}^n(t) + \frac{1}{n} \sum_{j=K+1}^n \Big(  \int_{(\tau,t]} \varphi^n(t-s) \,d N^n_j(s) - \int_{(\tau,t]} \varphi^n(t-s) \,d \mathcal{N}^n_j(s) \Big)
\end{align*}

By definition of $\tau_+$, there exists $t_1 > \tau$ such that $\lambda^n(t_1) < \widetilde{\lambda}^n(t_1)$. Using the last inequality we know that there exist $K+ 1 \leq j_1 \leq n$ and $\tau_+ < s_1 < t_1$ such that $\mathcal{N}^n_{j_1}$ jumps at $s_1$ and not $N^n_{j_1}$. By definition of $\tau$, we can also find $\tau_+ < t_2 < s_1$ such that $\lambda^n(t_2) < \widetilde{\lambda}^n(t_2)$. Then we can build iteratively the sequences $(t_k)_k$ and $(s_k)_k$ such that $\tau_+ < t_{k+1} < s_k < \tau_k$ and such that for all $k$, there exists $K+ 1 \leq j_k \leq n$ such that $\mathcal{N}^n_{j_k}$ jumps at time $s_k$. This implies that the jumps of the Hawkes process $(\widetilde{N}^n_{j})_{j}$ admits a finite accumulation point. This is impossible (a.s.) so $\tau_+ = \infty$ a.s. This proves that $\lambda_0^n(t) \geq \theta^n(t)$ for all $t \geq 0$. The two other inequalities follow directly from \eqref{eqn.410.0} and \eqref{eqn.411.0}.
  \qed 
 
 We now prove that as $n\to\infty$, the process $(\overline{\Lambda^{(n)}},\overline{N^{(n)}}, \overline{M^{(n)}},  N^{(n)}_1, M^{(n)}_1,\cdots, N^{(n)}_K, M^{(n)}_K)$ is asymptotically equivalent to 
 $( \overline{\Theta^{(n)}},\overline{\mathcal{N}^{(n)}}, \overline{\mathcal{M}^{(n)}},  \mathcal{N}^{(n)}_1, \mathcal{M}^{(n)}_1,\cdots, \mathcal{N}^{(n)}_K, \mathcal{M}^{(n)}_K)$ with 
 \beqnn
  \quad \overline{\Theta^{(n)}}:= \beta_n^2\cdot\overline{\Theta^n} ,\quad  \overline{\mathcal{N}^{n}} :=\beta_n^2\cdot\overline{\mathcal{N}^n} , \quad 
  \overline{\mathcal{M}^{n}} := \beta_n\cdot\overline{\mathcal{M}^n} ,   \quad 
   \mathcal{N}^{(n)}_i :=n\beta_n^2\cdot\mathcal{N}^n_i , \quad
 \mathcal{M}^{(n)}_i := \sqrt{n}\beta_n\cdot\mathcal{M}^n_i . 
 \eeqnn
To that extent, we first show that $n\beta_n^2\cdot\lambda^n_0 $ and $n\beta_n^2\cdot\theta^n$ are asymptotically equivalent as $n\to\infty$. 
  
 \begin{proposition}\label{Prop.403}
 	For each $T\geq 0$, there exists a constant $C>0$ such that for any $n\geq K$,
 	\beqlb\label{eqn.408}
 	\sup_{t\in[0,T]}n\beta_n \cdot \mathbf{E}\Big[\big| n\beta_n^2\cdot\lambda^n_0(t)-n\beta_n^2\cdot\theta^n(t) \big|\Big]
 	+ n\beta_n \cdot \mathbf{E}\Big[	\sup_{t\in[0,T]}\big| \overline{\Lambda^{(n)}}(t)-\overline{\Theta^{(n)}}(t) \big|\Big] \leq C. 
 	\eeqlb 
 \end{proposition}
 \proof Taking expectations on both sides of \eqref{eqn.410} and \eqref{eqn.411}, we have 
 \beqnn
 \mathbf{E}\big[ \lambda^n_0(t)-\theta^n(t) \big]  
 \ar=\ar  \int_0^t  \varphi^n(t-s)\mathbf{E}\big[ \lambda^n_0(s) \big]ds - \frac{n-L}{n}\int_0^t  \varphi^n(t-s)\mathbf{E}\big[ \theta^n(s) \big]ds\cr
 \ar=\ar \int_0^t  \varphi^n(t-s)\mathbf{E}\big[ \lambda^n_0(s)-\theta^n(s) \big]ds + \frac{L}{n}\int_0^t \varphi^n(t-s)\mathbf{E}\big[ \theta^n(s) \big]ds.
 \eeqnn
Therefore, $y(t) = \mathbf{E}[ \lambda^n_0(t)-\theta^n(t) ]$ satisfies a Volterra integral equation which can be solved explicitely using the the resolvent $\psi^n$ associated to $\varphi^n$ defined in \eqref{eqn.201}, see for instance Lemma~3 in \cite{BacryDelattreHoffmannMuzy2013}. Using then Fubini's theorem as well as the second equality in \eqref{eqn.201}, we get
 \beqnn
 \mathbf{E}\big[ \lambda^n_0(t)-\theta^n(t) \big]
 \ar=\ar\frac{L}{n}\int_0^t \varphi^n(t-s)\mathbf{E}\big[ \theta^n(s) \big]ds + \frac{L}{n}\int_0^t \psi^n(t-s) \int_0^s\varphi^n(s-r)\mathbf{E}\big[ \theta^n(r) \big]drds\cr
 \ar=\ar \frac{L}{n}\int_0^t \varphi^n(t-s)\mathbf{E}\big[ \theta^n(s) \big]ds  + \frac{L}{n}\int_0^t  \int_0^{t-s}\psi^n(t-s-r)\varphi^n(r)dr\mathbf{E}\big[ \theta^n(s) \big]ds\cr
 \ar=\ar \frac{L}{n}\int_0^t \psi^n(t-s)\mathbf{E}\big[ \theta^n(s) \big]ds .
 \eeqnn
 Using also Proposition~\ref{Prop.401} and Equation \eqref{eqn.401}, we have 
 \beqnn
 \sup_{t\in[0,T]}n\beta_n \cdot \mathbf{E}\Big[\big| n\beta_n^2\cdot\lambda^n_0(t)-n\beta_n^2\cdot\theta^n(t) \big|\Big] \leq 	  \int_0^T \beta_n \psi^n(s)ds \cdot \sup_{t\in[0,T]}\mathbf{E}\big[ n\beta_n^2\lambda^n_0(t)  \big] \leq C
 \eeqnn
 uniformly in $n\geq 1$. The uniform upper bound for $n\beta_n \cdot \mathbf{E}\big[	\sup_{t\in[0,T]}\big| \overline{\Lambda^{(n)}}(t)-\overline{\Theta^{(n)}}(t) \big|\big]$ follows directly from the previous result and the fact that 
 \beqnn
 \sup_{t\in[0,T]}\big| \overline{\Lambda^{(n)}}(t)-\overline{\Theta^{(n)}}(t) \big| 
 \leq \int_0^T \big| n\beta_n^2\cdot\lambda^n_0(t)-n\beta_n^2\cdot\theta^n(t) \big|dt. 
 \eeqnn
 \qed 
 
 \begin{proposition} \label{Prop.404}
 	For each $T\geq 0$, we have as $n\to\infty$,
 	\beqlb\label{eqn.412}
 	\mathbf{E}\bigg[ \sup_{t\in[0,T]} \Big| \overline{N^{(n)}}(t) - \overline{\mathcal{N}^{(n)}}(t)  \Big| \bigg] +  \mathbf{E}\bigg[ \sup_{t\in[0,T]} \Big| \overline{M^{(n)}}(t) - \overline{\mathcal{M}^{(n)}}(t)  \Big|^2 \bigg]  \to 0, 
 	\eeqlb
 	and for $1\leq i\leq K$, 
 	\beqlb \label{eqn.412.1}
 	\mathbf{E}\bigg[ \sup_{t\in[0,T]} \Big| N_i^{(n)}(t) - \mathcal{N}_i^{(n)}(t)  \Big| \bigg] +  \mathbf{E}\bigg[ \sup_{t\in[0,T]} \Big| M^{(n)}_i(t) -  \mathcal{M}^{(n)}_i(t)  \Big|^2 \bigg]  \to 0. 
 	\eeqlb
 \end{proposition}
 \proof We just prove \eqref{eqn.412} for conciseness, the limit \eqref{eqn.412.1} can be proved in the same way. 
 We first consider the second expectation in \eqref{eqn.412}. 
 In view of \eqref{eqn.211.c}, \eqref{eqn.410.0} and \eqref{eqn.413.c}, and using that $\theta^n\leq \lambda^n_0$ by Proposition~\ref{Prop.401}, we get 
 \beqnn
 \overline{M^{(n)}}(t) - \overline{\mathcal{M}^{(n)}}(t) 
 \ar=\ar \sum_{i=1}^d \int_0^t \int_0^{\lambda^n_0(s-)} \beta_n \widetilde{\Pi}^n_i(ds,dz)- \sum_{i=K+1}^n  \int_0^t \int_0^{\theta^n(s-)} \beta_n\widetilde{\Pi}^n_i(ds,dz)\cr
 \ar=\ar \sum_{i=1}^K \int_0^t \int_0^{\lambda^n_0(s-)} \beta_n \widetilde{\Pi}^n_i(ds,dz)+\sum_{i=K+1}^n  \int_0^t \int_{\theta^n(s-)}^{\lambda^n_0(s-)}\beta_n \widetilde{\Pi}^n_i(ds,dz).
 \eeqnn
 By the Burkholder-Davis-Gundy inequality, we have for some constant $C>0$ independent of $n$, 
 \beqnn
 \lefteqn{\mathbf{E}\bigg[ \sup_{t\in[0,T]} \Big| \overline{M^{(n)}}(t) - \overline{\mathcal{M}^{(n)}}(t)  \Big|^2 \bigg] 
 \leq C \cdot  \mathbf{E}\Big[   \big[  \overline{M^{(n)}} -  \overline{\mathcal{M}^{(n)}} \big]_T \Big]}\ar\ar\cr
 \ar=\ar C \cdot  \sum_{i=1}^K \mathbf{E}\bigg[   \Big[\int_0^\cdot \int_0^{\lambda^n_0(s-)} \beta_n \widetilde{\Pi}^n_i(ds,dz) \Big]_T \bigg] +C \cdot  \sum_{i=K+1}^n  \mathbf{E}\bigg[   \Big[ \int_0^\cdot \int_{\theta^n(s-)}^{\lambda^n_0(s-)} \beta_n \widetilde{\Pi}^n_i(ds,dz) \Big]_T \bigg]\cr
 \ar=\ar C\cdot  \sum_{i=1}^K \mathbf{E}\bigg[  \int_0^T \int_0^{\lambda^n_0(s-)} \beta_n^2 \Pi^n_i(ds,dz) \Big] \bigg] +C\cdot \sum_{i=K+1}^n  \mathbf{E}\bigg[     \int_0^T \int_{\theta^n(s-)}^{\lambda^n_0(s-)}\beta_n^2 \Pi^n_i(ds,dz) \bigg]\cr
 \ar=\ar   \frac{CK}{n}   \int_0^T  \mathbf{E}\big[n\beta_n^2\cdot\lambda^n_0(s)\big]  ds   + \frac{C (n-K)}{n}\int_0^T \mathbf{E}\Big[  \big|n\beta_n^2\cdot\lambda^n_0(s)-n\beta_n^2\cdot\theta^n(s)\big| \Big] ds,  
 \eeqnn
 which goes to $0$ as $n\to\infty$; see \eqref{eqn.401} and \eqref{eqn.408}.
 Here the first equality follows from the mutual independence of $\{\Pi^n_i(ds,dz)\}_{i\geq 1}$. 
 We now turn to consider the first expectation in \eqref{eqn.412}. By \eqref{eqn.209}, \eqref{eqn.211.b} and \eqref{eqn.413.a}-\eqref{eqn.413.c},
 \beqnn
 \sup_{t\in[0,T]} \Big| \overline{N^{(n)}}(t) - \overline{\mathcal{N}^{(n)}}(t)  \Big|
 \ar\leq\ar \beta_n  \cdot \sup_{t\in[0,T]} \Big|\beta_n  \cdot\overline{M^n}(t) -\beta_n  \cdot\overline{\mathcal{M}^n}(t)  \Big|\cr
 \ar\ar \quad + \int_0^T \big|n\beta_n^2  \lambda^n_0(s) -\beta_n^2 (n-K) \theta^n(s)  \big|ds\cr
 \ar\leq\ar \beta_n \cdot\sup_{t\in[0,T]} \Big|\overline{M^{(n)}}(t) - \overline{\mathcal{M}^{(n)}}(t)  \Big|\cr
 \ar\ar \quad + \int_0^T \big|n\beta_n^2  \lambda^n_0(s) -n\beta_n^2 \theta^n(s)  \big|ds 
 +K\int_0^T  \beta_n^2\theta^n(s)  ds
 \eeqnn
 and hence using also \eqref{eqn.408} and \eqref{eqn.401}
 \beqnn
 \mathbf{E}\bigg[ \sup_{t\in[0,T]} \Big| \overline{N^{(n)}}(t) - \overline{\mathcal{N}^{(n)}}(t)  \Big| \bigg]
 \ar\leq \ar
 \beta_n  \cdot \mathbf{E}\bigg[ \sup_{t\in[0,T]} \Big|\overline{M^{(n)}}(t) - \overline{\mathcal{M}^{(n)}}(t)  \Big| \bigg]\cr
 \ar\ar + \int_0^T \mathbf{E}\Big[\big|n\beta_n^2  \lambda^n_0(s) -n\beta_n^2 \theta^n(s)  \big|\Big]ds  \cr
 \ar\ar + \frac{K}{n}\int_0^T   
 \mathbf{E}\Big[ n\beta_n^2\theta^n(s)\Big]  ds  \to 0.
 \eeqnn
 This concludes the proof. 
 \qed 
 
  \textbf{\textit{Tightness and characterization.}}
We are now ready to conclude the proof of Theorem~\ref{Thm.MainResultSL.03} in the case $n\beta_n^2\to\zeta \in(0,\infty)$. By Prohorov's theorem and Propositions~\ref{Prop.403} and \ref{Prop.404}, the sequence 
 $\big\{( \overline{\Lambda^{(n)}}, \overline{N^{(n)}}, \overline{M^{(n)}},   N^{(n)}_1, M^{(n)}_1,\cdots ,  N^{(n)}_K , M^{(n)}_K \big)\big\}_{n\geq K}$
 is tight if and only if 
 \beqlb\label{eqn.423}
 \big\{\big( \overline{\Theta^{(n)}}, \overline{\mathcal{N}^{(n)}}, \overline{\mathcal{M}^{(n)}}, \mathcal{N}^{(n)}_1, \mathcal{M}^{(n)}_1,\cdots, \mathcal{N}^{(n)}_K, \mathcal{M}^{(n)}_K  \big) \big\}_{n\geq K}
 \eeqlb 
 is relatively compact.
 Hence it suffices to prove that any sequence 
 \beqlb\label{eqn.421}
 \big\{\big( \overline{\Theta^{(n_k)}}, \overline{\mathcal{N}^{(n_k)}}, \overline{\mathcal{M}^{(n_k)}}, \mathcal{N}^{(n_k)}_1, \mathcal{M}^{(n_k)}_1,\cdots,\mathcal{N}^{(n_k)}_K, \mathcal{M}^{(n_k)}_K \big) \big\}_{k\geq 1}
 \eeqlb
 contains a weakly convergent subsequence and the limit process satisfies \eqref{eqn.216c}. 
 By the $C$-tightness of $\{( \overline{\Lambda^{(n_k)}}, \overline{N^{(n_k)}}, \overline{M^{(n_k)}}) \}_{n\geq 1}$, we can assume that 
 \beqlb\label{eqn.422}
 (\overline{\Theta^{(n_k)}}, \overline{\mathcal{N}^{(n_k)}}, \overline{\mathcal{M}^{(n_k)}})\to (\overline{X},\overline{X},\overline{Z}),
 \eeqlb weakly in $\mathrm{D}(\mathbb{R}_+;\mathbb{R}_+^2\times\mathbb{R})$ as $k\to\infty$. 
 By the Skorokhod representation theorem and Proposition~1.17(b) in \cite[p.328]{JacodShiryaev2003}, we can also assume without loss of generality that the limit \eqref{eqn.422} holds uniformly on compacts and almost surely. 
 Finally, using Corollary~\ref{Coro.411} and the fact that $\overline{\Theta^{(n_k)}}=(n-K)\beta_n^2 \cdot \Theta^{n_k} $ we have
 \beqnn
 \mathcal{N}^{(n_k)}_i 
 \ar\overset{\rm d}=\ar n\beta_n^2\cdot N_i^\circ\circ \Big( \frac{\overline{\Theta^{(n_k)}}}{(n-K)\beta_n^2}   \Big) \overset{\rm a.s.} \to \zeta\cdot  N_i^\circ \circ   \frac{\overline{X}}{\zeta} 
 \eeqnn
and
 \beqnn
 \mathcal{M}^{n_k}_i 
 \ar\overset{\rm d}=\ar
 \sqrt{n}\beta_n\cdot \widetilde{N}^\circ_i\circ\Big( \frac{1}{n\beta_n^2} \cdot \beta_n^2 \overline{\Theta^{n_k}} \Big)
 \overset{\rm a.s.} \to \sqrt{\zeta}\cdot \widetilde{N}^\circ_i \circ  \frac{\overline{X}}{\zeta},
 \eeqnn
 as $k\to\infty$ uniformly on compacts for each $1\leq i\leq K$. 
 Consequently, the sequence \eqref{eqn.421} converges weakly and hence \eqref{eqn.423} is relatively compact and any limit process is a weak solution of \eqref{eqn.216c}.

  \subsection{Case~III: $n\beta_n^2\to \zeta=\infty$} 
 \label{sec:proof2:III}
 
In this case, we want to prove that $X_i = Z_i = 0$ almost surely. By definition, it suffices to prove that for any $T\geq 0$ and $\epsilon>0$,
 \beqlb\label{eqn.414}
 \mathbf{P}\Big\{ N^{(n)}_i(T) \geq \epsilon\Big\} + \mathbf{P}\bigg\{\sup_{t\in[0,T]}  \big| M^{(n)}_i(t)\big| \geq \epsilon\bigg\} \to 0, \quad 1\leq i\leq K,
 \eeqlb
 as $n\to\infty$. 
Since $N^n_i(T)$ is an integer we see that
\begin{equation*}
\mathbf{P}\Big\{ N^{(n)}_i(T) \geq \epsilon\Big\} \leq \mathbf{P}\big\{  N^n_i(T) \geq 1\big\}.
\end{equation*}
Using then Chebyshev's inequality, we have
 \beqnn
 \mathbf{P}\big\{  N^n_i(T) \geq 1\big\}
  \leq \frac{1}{n\beta_n^2}\cdot\mathbf{E}\Big[n\beta_n^2 \cdot N^n_i(T) \Big]= \frac{\mathbf{E}\big[ \Lambda_0^{(n)}(T) \big]}{n\beta_n^2},
 \eeqnn
 which goes to $0$ as $n\to\infty$. (by  \eqref{eqn.401}). On the other hand, since $M^{n}_i(t)=  N^n_i(t) -\Lambda^n_0(t) $, then
 \beqnn
 \sup_{t\in[0,T]} \big| M^{(n)}_i(t)\big| \leq  \sqrt{n}\beta_n  \cdot   N^n_i(T)  +\sqrt{n}\beta_n  \cdot   \Lambda^n_0(T).
 \eeqnn
 The preceding result along with the fact that $\sqrt{n}\beta_n \to \infty$ implies that 
 \beqnn
 \sqrt{n}\beta_n  \cdot   N^n_i(T) \leq    N^{(n)}_i(T) \overset{\rm p}\to 0.
 \eeqnn
 Moreover, by \eqref{eqn.401} we also have as $n\to\infty$,
 \beqnn
 \mathbf{E}\Big[\sqrt{n}\beta_n  \cdot   \Lambda^n_0(T)\Big] 
 = \frac{1}{\sqrt{n}\beta_n}\cdot \mathbf{E}\Big[   \Lambda^{(n)}_0(T)\Big]
 \leq    \frac{C}{\sqrt{n}\beta_n}  \to 0.
 \eeqnn
 Combining all these estimates together yields \eqref{eqn.414}.

 \section{Proof of Theorem~\ref{Main.Thm03}}
 \label{Sec.Proof03} 
 \setcounter{equation}{0}
 
Before proving Theorem~\ref{Main.Thm03}, we first state the following key lemma, which establishes a weak convergence result for empirical distributions of a sequence of interchangeable random variables. 
 Let $\mathbb{S}$ be a separable and complete metric space endowed with the $\sigma$-algebra $\mathscr{S}$.

 \begin{lemma} \label{Lemma.501}
 Let $\overline{\xi},\xi_1, \xi_2, \cdots$ be a sequence of $\mathbb{S}$-valued random variables satisfying that conditioned on $\overline{\xi}$, the random variables $\xi_1,\xi_2,\cdots$ are  i.i.d  with common distribution $Q_{\overline{\xi}}$ on $\mathbb{S}$. 
 Consider a sequence of exchangeable $\mathbb{S}$-valued random variables $\{(\overline{\xi^n}, \xi^n_1, \dots, \xi^n_n)\}_{n\geq 1}$, i.e., for each $n \geq 1$ and permutation $\sigma$ of $\{1, \dots, n \}$, 
 \beqlb\label{eqn.501}
 \big(\overline{\xi^n}, \xi^n_{\sigma(1)}, \dots, \xi^n_{\sigma(n)}\big)\overset{\rm d}=\big(\overline{\xi^n}, \xi^n_1, \dots, \xi^n_n\big).
 \eeqlb
 If  $(\overline{\xi^n},\xi^n_1, \dots, \xi^n_K) \overset{\rm d}\to (\overline{\xi}, \xi_1, \dots, \xi_K)$ as $n\to\infty$ for each $K \geq 1$, then 
 \beqlb\label{eqn.505}
 P^n_\xi := \frac{1}{n} \sum_{k=1}^n \delta_{\xi^n_k} \overset{\rm d}\to Q_{\overline{\xi}}.
 \eeqlb
 \end{lemma}
 
 The proof of this lemma relies on the next auxiliary result about the weak convergence of random measures, which can be obtained directly from Theorems 4.11 and 4.19 in \cite[p.111, 126]{Kallenberg2017}. Recall that $\nu(g)$ denotes the integral of function $g$ with respect to the measure $\nu$.

 \begin{proposition}\label{Prop.Lemma.502}
 For a sequence of bounded random measures $\pi, \pi_1, \pi_2, \dots$ on $\mathbb{S}$, we have 
 $\pi_n\overset{\rm d}\to \pi $ as $n\to\infty$ if and only if $\pi_n(g)\overset{\rm d}\to \pi(g)$ for any continuous function $g$ on $\mathbb{S}$ with compact support.   
 \end{proposition}
 
 \textit{Proof of Lemma~\ref{Lemma.501}.} 
 By Proposition~\ref{Prop.Lemma.502}, it suffices to prove that 
 \beqlb\label{eqn.502}
 P^n_\xi(g) = \frac{1}{n} \sum_{k=1}^n g(\xi^n_k) \overset{\rm d}\to Q_{\overline{\xi}}(g) =  \int_\mathbb{S} g(x) Q_{\overline{\xi}}(dx),
 \eeqlb
 as $n\to\infty$ for any continuous function $g$ on $\mathbb{S}$ with compact support.
 The continuity and compact support of $g$ yield that 
 \beqnn 
 \sup_{n\geq 1} \big| P^n_\xi(g) \big| \leq \big\|g\big\|_{L^\infty}<\infty,\quad a.s. 
 \eeqnn
 By  Theorem 30.2 in \cite{Billingsley2008}, the limit \eqref{eqn.502} follows if  for any $K\geq 1$,
 \beqlb
 	\mathbf{E}\Big[\big( P^n_\xi(g)\big)^K\Big] \to \mathbf{E}\Big[ \big( Q_{\overline{\xi}}(g)\big)^K\Big].
 \eeqlb 
 Indeed, the exchangeability of $\{\xi^n_i\}_{1\leq i\leq n}$ ensures that
 \beqlb\label{eqn.504}
 \mathbf{E}\Big[\big( P^n_\xi(g)\big)^K\Big]
 \ar=\ar \mathbf{E}\bigg[\Big( \frac{1}{n} \sum_{k=1}^n g(\xi^n_k)\Big)^K\bigg]
 \cr \ar=\ar \frac{1}{n^K} \sum_{k_1, \dots, k_K=1}^n \mathbf{E}\bigg[\prod_{j=1}^K g(\xi^n_{k_j}) \bigg]\cr
 \ar=\ar \frac{n!}{(n-K)! \, n^K} \mathbf{E}\bigg[\prod_{j=1}^K g(\xi^n_{j}) \bigg]+ 
 \frac{1}{n^K} \sum_{(*)} \mathbf{E}\bigg[\prod_{j=1}^K g(\xi^n_{k_j})\bigg]
 \eeqlb 
 where the last sum is over all indexes $1 \leq k_1, \dots, k_K \leq n$ with $k_i =k_j$ for some $i \neq j$. 
 The assumption $(\overline{\xi^n},\xi^n_1, \dots, \xi^n_K) \overset{\rm d}\to (\overline{\xi}, \xi_1, \dots, \xi_K)$ implies that 
 \beqnn 
  \mathbf{E}\bigg[\prod_{j=1}^K g(\xi^n_{k_j})\bigg] 
  \to \mathbf{E}\bigg[\prod_{j=1}^K g(\xi_{j})\bigg] = \mathbf{E}\bigg[\Big(\mathbf{E}\Big[ g(\xi_1) \,\big|\,\overline{\xi} \Big]\Big)^K\bigg]= \mathbf{E}\Big[ \big(Q_{\overline{\xi}}(g)\big)^K\Big]
 \eeqnn
 as $n\to\infty$. Here the last two equalities follow from the conditional independence and identical distribution of $\xi_1,\cdots,\xi_K$. 
 Additionally, the boundedness of $g$ yields  
 \beqnn
 \mathbf{E}\bigg[\prod_{j=1}^K g(\xi^n_{k_j})\bigg] \leq \big\|g\big\|_{L^\infty}^K
 \quad \mbox{and hence}\quad
 \frac{1}{n^K} \sum_{(*)} \mathbf{E}\bigg[\prod_{j=1}^K g(\xi^n_{k_j})\bigg] 
 \leq C\cdot \Big(1- \frac{n!}{(n-K)! \, n^K}  \Big),
 \eeqnn
 uniformly in $n\geq K$. Note that as $n\to\infty$,
 \beqnn
 \frac{n!}{(n-K)! \, n^K} \to 1.
 \eeqnn
 Plugging this into \eqref{eqn.504} immediately proves that 
 \beqnn
 \lim_{n\to\infty}\mathbf{E}\Big[\big( P^n_\xi(g)\big)^K\Big] 
 = \mathbf{E}\Big[ \big(Q_{\overline{\xi}}(g)\big)^K\Big]
 \eeqnn
 and hence \eqref{eqn.505} holds.
 \qed 
 
 {\it Proof of Theorem~\ref{Main.Thm03}.} 
 Since $\mathrm{D}(\mathbb{R}_+; \mathbb{R}_+\times \mathbb{R})$ is a separable and complete metric space, it suffices to identify that $\mathrm{D}(\mathbb{R}_+; \mathbb{R}_+\times \mathbb{R})$-valued random variables 
 \beqnn
 \big((\overline{X}, \overline{Z}), (X_1,Z_1),(X_2,Z_2), \cdots  \big)
 \quad\mbox{and}\quad 
 \big\{ \big((\overline{N^{(n)}}, \overline{M^{(n)}}), (N^{(n)}_1, M^{(n)}_1),\cdots, (N^{(n)}_n, M^{(n)}_n) \big) \big\}_{n\geq 1}
 \eeqnn
  satisfy assumptions in Lemma~\ref{Lemma.501}. 
  Indeed, in Theorem~\ref{Thm.MainResultSL.03}(2), the Brownian motions $W_1,\cdots,W_K$ and Poisson processes $N^\circ_1,\cdots,N^\circ_K$ are mutually independent and they are also independent of $(\overline{X},\overline{Z},B)$. Therefore, conditioned on $(\overline{X},\overline{Z})$, the random variables $(X_1,Z_1),(X_2,Z_2), \cdots$ are i.i.d. with common distribution $\delta_{\overline{X}}\otimes\mathcal{N}\circ \overline{X}$ when $\zeta =0$, $\mathcal{P}_\zeta \circ \frac{\overline{X}}{\zeta}$ when $\zeta \in(0,\infty)$ or  $\delta_{(0,0)}$ when $\zeta =\infty$. 
  On the other hand, for each $n\geq 1$, the exchangeability of the particle system $(N^n_1,\cdots,N_n^n)$ immediately induces that \eqref{eqn.501} holds for $\big((\overline{N^{(n)}}, \overline{M^{(n)}}), (N^{(n)}_1, M^{(n)}_1),\cdots, (N^{(n)}_n, M^{(n)}_n) \big) $. 
  Finally, by Lemma~\ref{Lemma.501}, we have the weak convergence 
  \beqnn
  \big((\overline{N^{(n_k)}}, \overline{M^{(n_k)}}), (N^{(n_k)}_1, M^{(n_k)}_1),\cdots, (N^{(n_k)}_K, M^{(n_k)}_K) \big)
   \to 
 \big((\overline{X}, \overline{Z}), (X_1,Z_1),\cdots,(X_K,Z_K)  \big)
   \eeqnn
   in $\mathrm{D}(\mathbb{R}_+;(\mathbb{R}_+\times \mathbb{R})^{1+K})$. This implies the desired weak convergence \eqref{eqn.213} with the limit given as before. 
   \qed

 \bibliographystyle{plain}

 \bibliography{Reference}

 \appendix
 
 \renewcommand{\theequation}{A.\arabic{equation}} 
 \section{Summary about Hawkes processes}
 \label{Sec.AppendixHawkes} 
 \setcounter{equation}{0}

In this section, we summarize some results about multidimensional Hawkes processes. 

\begin{definition}
\label{def:hawkes}
 A {\rm$d$-dimensional Hawkes process} with {\rm baseline} (or {\rm background rate}) $\mu \in \mathbb{R}_+^d$ and {\rm self-exciting kernel} $\varphi : \mathbb{R}_+ \to \mathbb{R}_+^{d\times d} $ is a $d$-dimensional $(\mathscr{F}_t)$-adapted point process $N = (N_1, \dots, N_d)$ such that
\begin{itemize}
\item Almost surely, for all $i \neq j$, the two point processes $N_i$ and $N_j$ never jump simultaneously,
\item For every $i$, the compensator $\Lambda_i$ of $N_i$ has the form $\Lambda_i(t) = \int_0^t \lambda_i(s)\, ds$ with
\begin{equation}
\label{eq:def:intensity}
 \lambda_i(t) = \mu_i + \sum_{j=1}^d \int_0^{t} \varphi_{i,j}(t-s) \,dN_j(s),\quad   t \geq 0.
\end{equation}
\end{itemize}
\end{definition}

Alternatively, Hawkes processes also can be defined through their Poisson representation.

\begin{definition}
\label{def:hawkes:1}
 Consider $d$ independent Poisson random measures $  \Pi_i(ds,dz)  $, $i=1,\cdots,d $, on $\mathbb{R}_+^2$ with intensity $ds\,dz$. 
 A {\rm$d$-dimensional Hawkes process} with {\rm baseline} (or {\rm background rate}) $\mu \in \mathbb{R}_+^d$ and {\rm self-exciting kernel} $\varphi : \mathbb{R}_+ \to \mathbb{R}_+^{d\times d} $ is a $d$-dimensional $(\mathscr{F}_t)$-adapted point process $N = (N_1, \dots, N_d)$ such that
\beqlb
\label{eq:poissonhawkes}
N_i(t)= \int_0^t \int_0^{\lambda_i(s-)} \Pi_i(ds,dz) , \quad t\geq 0,
\eeqlb
where $\lambda_i$ is given as in \eqref{eq:def:intensity}.
\end{definition}

Equivalence between Definitions \ref{def:hawkes} and \ref{def:hawkes:1} is done (for instance) in Theorem~7.4 in \cite[p.90]{IkedaWatanabe1989}.
\begin{theorem}[Poisson representation of Hawkes processes]
\begin{enumerate}
\item Suppose that $N$ is a Hawkes process in the sense of Definition \ref{def:hawkes:1}. Then it is also a Hawkes process in the sense of Definition \ref{def:hawkes}
\item Suppose that $N$ is a Hawkes process in the sense of Definition \ref{def:hawkes}. Then we can build on an enlarged probability space $(\widetilde {\Omega}, \widetilde{ \mathscr{F}}, \widetilde{ \mathscr{F}}_t, \widetilde{\mathbf{P}})$ a family of independent Poisson measures $\{ \Pi_i(ds,dz) \}_{1\leq i\leq d}$ on $\mathbb{R}_+^2$ with intensity $ds\,dz$ such that \eqref{eq:poissonhawkes} holds.
\end{enumerate}
\end{theorem}

For convenience, we write  \eqref{eqn.Intensity} in the following matrix form
 \beqlb
 \lambda(t)=\mu +  \int_0^t \phi(t-s)\, dN(s),\quad t\geq 0.
 \eeqlb
 Note that $N$ has compensator $\Lambda=(\Lambda_1,\cdots,\Lambda_d)$ with
 \beqlb 
 \Lambda_i (t) := \big\|\lambda_i\big\|_{L^1_t}=\int_0^t \lambda_i(s) \, ds,\quad t\geq 0,\, i=1,\cdots, d
 \eeqlb
 and we define the compensated point process 
 $M:=\{M(t):t\geq 0\}$ with 
 \begin{equation*}
 M(t) := N(t) - \Lambda(t),
 \end{equation*}
 which is a $(\mathscr{F}_t)$-martingale. Since components of $N$ never jump simultaneously, $M$ has predictable quadratic covariation and quadratic covariation 
 \beqlb\label{eqn.211}
  \big\langle M_i,M_j \big\rangle = \mathbf{1}_{\{ i=j\}} \cdot \Lambda_i
 \quad \mbox{and}\quad 
  \big[M_i,M_j \big] = \mathbf{1}_{\{ i=j\}} \cdot N_i,\quad 1\leq i\leq d. 
 \eeqlb
 Additionally, we define the resolvent of the second kind associated to $\phi$  by
 \beqlb\label{Resolvent}
 \psi(t) := \sum_{k\geq 1} \phi^{*k}(t),\quad t\geq 0,
 \eeqlb
 which satisfies the resolvent equation
 \beqlb\label{eqn.Resolvent}
 \psi(t)=\phi(t) + \psi*\phi(t),\quad t \geq 0.
 \eeqlb 
 From this equation, we obtain the following powerful martingale representation of the intensity $\lambda$, see for instance Lemma~3 in \cite{BacryDelattreHoffmannMuzy2013} or Proposition~2.1 in \cite{JaissonRosenbaum2015}. 
 
 \begin{lemma}[Martingale representation]
 	The intensity $\lambda$ is the unique solution to 
 	\beqlb\label{MartingaleRep}
 	 \lambda(t)= \mu + \big\|\psi\big\|_{L^1_t}\cdot \mu +\int_0^t \psi(t-s)\,dN(t), \quad t\geq 0.
 	\eeqlb
 \end{lemma}

 \renewcommand{\theequation}{B.\arabic{equation}} 
 
 \section{Regularity of $(X,Z,Y)$}
 \label{Sec.AppendixHCY} 
 \setcounter{equation}{0}

 In this section, we consider the regularity of the solution $(X,Z)$ of (\ref{eqn.LimitSVE}) and the derivative $Y$ of $X$ on the interval $[0,1]$ with $d=1$; see Theorem~\ref{Thm.MainResultSL.01} and Remark~\ref{Remark.01}. All the following results can be easily generalized to the general case with $d>1$ and interval $[0,T]$ for $T\geq 0$. 
 Without further mention, we always assume that functions and processes in this section are restricted on $[0,1]$. We consider $p\in(0,\infty]$ and $\alpha \in (0,1]$, we write $L^p=L^p([0,1];\mathbb{R})$ and we denote by $H^\alpha$ the space of all $\alpha$-H\"older continuous functions on $(0,1]$.

We first recall some basic properties of  convolution operator  as well as fractional integral and fractional derivative. 
 The fractional integral of order $\alpha\in(0,1]$ and fractional derivative of order $\gamma\in [0,1)$ of a function $g$ are defined respectively by
 \beqnn
 I^\alpha g(t) = \frac{1}{\Gamma ( \alpha )}
 \int_0^t \frac{g(s)}{(t-s )^{1-\alpha}} \, ds  
 \quad \text{and}\quad 
  D^\gamma g(t) = \frac{1}{\Gamma (1- \gamma)}
 \frac{d}{dt} \int_0^t\frac{g(s)}{( t-s ) ^{\gamma}} \,ds  ,
 \eeqnn
 whenever they exist. Here $\Gamma(\cdot)$ is the Gamma function. 
 Without confusion, we write $D^{-\alpha}=I^\alpha$ and $I^{-\gamma}=D^\gamma$. 
 The following results come from Appendix A.1 in \cite{JaissonRosenbaum2016}.
 
  \begin{proposition}\label{Prop.B01}
 	\begin{enumerate}
 		\item[(1)] For $g_1 \in H^\alpha$ and $g_2\in L^\infty$, we have $g_1 *g_2 \in H^\alpha$. 
 		
 		\item[(2)] For $g_1 \in H^{\alpha_1}$ and $g_2  \in H^{\alpha_2}$, Then $g_1 * g_2 \in H^{\gamma\wedge 1}$ for any $0<\gamma< \alpha_1+\alpha_2$.
 		
 		\item[(3)] Assume that $g_1$ is continuous and $g_2$ is differentiable on $(0,1]$. 
 		For some constants $\alpha \in (0,1)$ and $C>0$, if $|g_2(t)| \leq C \cdot t^{-\alpha}$ and $|g'_2(t)| \leq C\cdot  t^{-\alpha-1}$ for any $t\in (0,1]$, then $g_1 *g_2 \in H^{1-\alpha}$.  
 		
 		\item[(4)] For $g_1 \in H^\alpha $ with $g_1(0)=0$ and $g_2 \in  L^p $ with $p>1$, we have $D^\beta g_1 \in H^{\alpha-\beta} $ and  $g_1*g_2 = (D^\beta g_1) * (I^\beta g_2)$ for any $0\leq \beta<\alpha$.
 		
 		\item[(5)]  For $g_1\in L^p $ and $g_2\in  L^q $ with $p,q\in(0,\infty]$, we have that $I^\alpha g_1 $ and $I^\alpha g_2$ exist, and $g_1* (I^\alpha g_2)= (I^\alpha g_1 ) * g_2 $ for any $(1/p+1/q-1)^+<  \alpha\leq 1$, .

 		\item[(6)] Let $g_1\in L^\infty$ and $g_2$ such that $ t^\theta g_2(t)\in H^\alpha $ for some $\theta\in (0,1)$.  
 		For any $0<\beta_1<(1-\theta)\wedge \alpha$ and $0<\beta_2\leq \alpha$, we have $D^\beta g_2 \in L^p $ for some $p>1$ and  $g_1*g_2 = (I^\alpha g_1) * (D^\alpha g_2)$.

 	\end{enumerate}
 \end{proposition}

 We now consider the regularity of the solution $(X,Z)$ of \eqref{eqn.LimitSVE}.
 For $\alpha \in (0,1]$, let $\mathcal{A}^\alpha$ be the collection of all differentiable functions $g$ on $(0,1]$ satisfying that $ t^{1-\alpha} g(t) \in H^\alpha $ and for any $ \gamma \in(-1/2,\alpha)$, both $D^\gamma g$ and $D^\gamma g'$ exist and uniformly in $t\in(0,1]$, 
 \beqlb \label{eqn.B01}
  \big| D^\gamma g(t) \big|+\big|t\cdot  (D^\gamma g)'(t)\big| \leq C\cdot t^{\alpha- 1 - \gamma}.
  \eeqlb
  
 \begin{lemma}
 If $f \in H^\alpha$ or $f \in\mathcal{A}^\alpha$ for some $\alpha \in (1/2,1]$, then $X\in H^1 $ and $Z\in H^{1/2-\epsilon}$  for any $\epsilon\in(0,1/2)$. 	
 \end{lemma}
 \proof Using the continuity of $Z$, Proposition~\ref{Prop.B01}(1) if $f \in H^\alpha$ or Proposition~\ref{Prop.B01}(3) if $f \in \mathcal{A}^\alpha$, we immediately get that $X \in H^\alpha$. Since the Brownian motion $B \in H^{1/2-\epsilon}$, we also have $Z\in H^{\alpha/2-\epsilon}$ for any $\epsilon\in(0,\alpha/2)$ and the proof is completed when $\alpha=1$. 
 If $\alpha\in (1/2,1)$, the result follows the following statement 
 \beqlb\label{eqn.B02}
 \mbox{if $X\in H^\beta$ for some $\beta\in (0,1)$, then $ X \in H^{\gamma}$ for any $\gamma<(\alpha+\beta/2)\wedge 1$}.
 \eeqlb 
  
  We now prove \eqref{eqn.B02}. 
 If $f \in H^\alpha$, by Proposition~\ref{Prop.B01}(2) we have $f*Z \in H^{(\alpha+\beta/2-\epsilon)\wedge 1}$ and hence \eqref{eqn.B02} holds. 
 On the other hand, if $f \in \mathcal{A}^\alpha$, the factional derivative $ D^{\beta/2-\epsilon} Z$ exists and is continuous by Proposition~\ref{Prop.B01}(4). 
 Using also \eqref{eqn.B01}, we have 
 \beqnn
 \big|I^{\beta/2-\epsilon}f(t)\big| \leq C\cdot t^{\alpha+\beta/2-1-\epsilon} 
 \quad \mbox{and}\quad 
 \big|  (I^{\beta/2-\epsilon}f)'(t)  \big|\leq C\cdot t^{\alpha+\beta/2-2-\epsilon} ,
 \eeqnn
 which, together with Proposition~\ref{Prop.B01}(6), implies that $f*Z = \big(I^{\beta/2-\epsilon}f\big) * \big( D^{\beta/2-\epsilon} Z\big)  \in H^{\alpha+\beta/2-\epsilon}$ for any $\epsilon \in(0,\beta/2)$,
 and thus \ref{eqn.B02} holds. 
 \qed
 
 \begin{lemma}
 	If $f \in H^\alpha$ or $f \in\mathcal{A}^\alpha$ for some $\alpha \in (1/2,1]$, then $Y \in H^{\alpha-1/2-\epsilon}$ for any $\epsilon \in (0,\alpha-1/2)$. 
 	
 \end{lemma}
 \proof For any $ \gamma\in(1-\alpha, 1/2)  $, the fractional derivative $D^{1-\gamma}f$ exists and belongs to $H^{\alpha-\gamma}$ when $f \in H^\alpha$ (see Proposition~\ref{Prop.B01}(4)) or $L^p$ with $p>1$ when $f \in \mathcal{A}^\alpha$ (see Proposition~\ref{Prop.B01}(6)). 
 Additionally, by the definition we have  
 \beqnn
 I^{1-\gamma} Z(t) = \int_0^t D^{\gamma}Z(s) \, ds,\quad t\geq 0.
 \eeqnn 
 Using Proposition~\ref{Prop.B01}(4) again and then Fubini's theorem, we have for any $t \in[0,1]$, 
 \beqnn
 f*Z(t) = \big(D^{1-\gamma} f \big) * \big(I^{1-\gamma} Z \big)(t) 
 \ar=\ar	\int_0^t D^{1-\gamma}f(t-s) \, ds
 \int_0^s D^{\gamma} Z(r)
 \, dr   \cr
 \ar=\ar \int_0^t (D^{1-\gamma}f)* (D^{\gamma} Z)(s)\, ds.
 \eeqnn 
 Plugging this back into \eqref{eqn.LimitSVE} yields that 
 \beqnn
 X(t) = \int_0^t Y(s)\, ds
 \quad \mbox{with}\quad 
 Y(t)=  F(t) \cdot a    +  (D^{1-\gamma}f)* (D^{\gamma} Z)(t) ,\quad t\geq 0. 
 \eeqnn
 Since $F$ is Lipschitz continuous, it suffices to prove that $ (D^{1-\gamma}f)* (D^{\gamma} Z) \in H^{\alpha-1/2-\epsilon}$.
 Indeed, when $f \in H^\alpha$, it follows directly from Proposition~\ref{Prop.B01}(2) along with the two facts that $D^{1-\gamma}f\in H^{\alpha-\gamma}$ and $D^{\gamma} Z \in H^{1/2-\gamma-\epsilon}$. 
 When $f \in \mathcal{A}^\alpha$, we have 
 \beqnn
 \big|D^{1-\gamma}f(t)\big|  \leq C\cdot t^{ \alpha+\gamma-2} 
 \quad \mbox{and}\quad 
 \big|  (D^{1-\gamma}f)'(t)   \big|\leq C\cdot t^{ \alpha+\gamma-3},
 \eeqnn
 which, together with Proposition~\ref{Prop.B01}(6) and the continuity of $D^{\gamma} Z$, induces that $(D^{1-\gamma}f)* (D^{\gamma} Z) \in H^{ \alpha+\gamma-1}$.
 Finally, the arbitrariness of $ \gamma\in(1-\alpha, 1/2)$ shows that $(D^{1-\gamma}f)* (D^{\gamma} Z) \in H^{\alpha-1/2-\epsilon}$.
 \qed

 \end{document}